\documentclass[10pt]{amsart}
\usepackage{latexsym,amsmath,amssymb}
\usepackage{mathrsfs}
\usepackage{mathabx}
 \usepackage{indentfirst}
 \usepackage{galois}
 \usepackage{color}
 \renewcommand{\epsilon}{\varepsilon}
 \newcommand{\newsection}[1]
  {\section{#1}\setcounter{theorem}{0} \setcounter{equation}{0}\par\noindent}

   \newtheorem{theorem}{Theorem}[section]
   \newtheorem{lemma}[theorem]{Lemma}
 \newtheorem{corr}[theorem]{Corollary}
 
 \newtheorem{proposition}[theorem]{Proposition}
 \newtheorem{deff}[theorem]{Definition}
 \newtheorem{remark}[theorem]{Remark}

  \numberwithin{equation}{section}

 \newcommand{\bth}{\begin{theorem}}
 \newcommand{\ble}{\begin{lemma}}
 \newcommand{\bcor}{\begin{corr}}
 \newcommand{\bdeff}{\begin{deff}}
 \newcommand{\bprop}{\begin{proposition}}
 \def\be{\begin{equation}}
\def\ee{\end{equation}}
\def\bt{\begin{theorem}}
\def\et{\end{theorem}}
\def\ba{\begin{array}}
\def\ea{\end{array}}
\def\bl{\begin{lemma}}
\def\el{\end{lemma}}

 \newcommand{\ele}{\end{lemma}}
 \newcommand{\ecor}{\end{corr}}
 \newcommand{\edeff}{\end{deff}}
 
 \newcommand{\eprop}{\end{proposition}}

 \renewcommand{\Pi}{\varPi}

 \renewcommand{\epsilon}{\varepsilon}

\title[Global regularity for Einstein-Klein-Gordon system] { Global regularity for Einstein-Klein-Gordon system with $U(1) \times \mathbb{R}$ isometry group, \uppercase\expandafter{\romannumeral2}}

\date{\today}

\begin{document}
\maketitle

\centerline{
\author{Haoyang Chen
  \footnote{School of Mathematical Sciences, Fudan University, Shanghai, China. {\it Email: hy{\textunderscore}chen15@fudan.edu.cn }}
 }
 \and
Yi Zhou
  \footnote{ *Corresponding Author: School of Mathematical Sciences, Fudan University, Shanghai, China.
 {\it Email: yizhou@fudan.edu.cn}
 }
  }

\begin{abstract}
This paper is devoted to the study of the global existence of smooth solutions for the 3+1 dimensional Einstein-Klein-Gordon systems with a $U(1) \times \mathbb{R}$ isometry group for a class of regular Cauchy data. In our first paper \cite{chen}, we reduce the Einstein equations to a 2+1 dimensional Einstein-wave-Klein-Gordon system. And we show that the first possible singularity can only occur at the axis. In this paper, we give a proof for the global regularity for the 2+1 dimensional system. Firstly, we show the non-concentration of the energy near the first possible singularity. Then, we prove that the global regularity holds for initial data with small energy.
\end{abstract}


\newsection{Introduction}
\subsection{Introduction and previous results}
Let $(^{(4)}{M},^{(4)}{g})$ be a 3+1 dimensional globally hyperbolic Lorentzian manifold which satisfies the following Einstein-scalar field equations:
\begin{equation}\label{1.1}
\begin{cases}
^{(4)}{G}_{\mu \nu}=^{(4)}{T}_{\mu\nu}=\partial_\mu \phi \partial_\nu \phi-\frac{1}{2} {^{(4)}{g}}_{\mu \nu}\partial^{\lambda} \phi \partial_{\lambda} \phi-{^{(4)}{g}}_{\mu\nu}V(\phi) \\
\Box_{^{(4)}{g}} \phi=V'(\phi)
\end{cases}
\end{equation}
where $^{(4)}{g}$ is the Lorentzian metric, $^{(4)}{G}_{\mu \nu}$ is the Einstein tensor of $^{(4)}{g}$, $^{(4)}{T}_{\mu\nu}$ is the stress-energy tensor given as above, $\phi$ is the scalar field, and we take the potential $V(\phi)=\frac{1}{2}m^2 \phi^2$.

Then, the equation that the scalar field satisfies is a Klein-Gordon equation, which makes the system an Einstein-Klein-Gordon system. The Einstein-Klein-Gordon system \eqref{1.1}  is equivalent to the following equations,
\begin{equation}\label{1.2}
\begin{cases}
^{(4)}{R}_{\mu \nu}=^{(4)}{\rho}_{\mu\nu} \triangleq \partial_\mu \phi \partial_\nu \phi+{^{(4)}{g}}_{\mu\nu} \frac{1}{2}m^2 \phi^2\\
\Box_{^{(4)}{g}} \phi=m^2 \phi
\end{cases}
\end{equation}
where $^{(4)}{R}_{\mu \nu}$ is the Ricci curvature tensor.

One fundamental open problem in the field of general relativity is the cosmic censorship conjectures by Penrose. Roughly speaking, the weak cosmic censorship may be formulated as follows: For generic asymptotically flat Cauchy data, of the vacuum equations or suitable Einstein-matter systems, the maximal development possesses a complete future null infinity. While the strong cosmic censorship states that the maximal Cauchy development is inextendible for generic initial data.

This question is partially related to the study of the formation of event horizons. Although still open in general, there is a series of results in spherical symmetric case by Christodoulou for the Einstein-scalar field system where the scalar field is massless, see \cite{christodoulou3}, \cite{christodoulou6}, etc.

 On the other hand, of our interest, it is related to the global well-posedness of the Cauchy problem in large. In \cite{christodoulou1}, Christodoulou proved the global existence of classical solutions for Einstein's equations in the spherically symmetric case with a massless scalar field on condition that the initial data are sufficiently small. Then, the smallness condition was removed in \cite{christodoulou5} and it was shown that a generalized solution exist globally in retarded time. The global nonlinear stability of Minkowski spacetime, for vacuum Einstein equations, was shown by Christodoulou and Klainerman in \cite{christodoulou4}. Lindblad and Rodnianski have proved the global stability of Minkowski spacetime for a massless Einstein-scalar field system in harmonic gauge in \cite{lindblad2}. Then in \cite{lefloch}, Lefloch and Ma show the nonlinear stability of Minkowski spacetime for the Einstein-Klein-Gordon system.

Based on \cite{chen}, this paper is devoted to study the global existence of smooth solutions for the 3+1 dimensional Einstein-Klein-Gordon systems with a $U(1) \times \mathbb{R}$ isometry group. The present work in this paper is motivated by research on the vacuum Einstein equations which is related to the study of wave map systems. We first provide some related results on the vacuum case.

For the 3+1 dimensional vacuum Einstein equations with one spacelike Killing field, as we can see in \cite{choquet3} and \cite{Moncrief1}, the Einstein equations can be reduced to a 2+1 dimensional Einstein-wave map system on a 2+1 dimensional Lorentzian manifold where the target manifold is the hyperbolic space $\mathbb{H}^2$. Yet the global existence problem of the Einstein-wave map system is still open. As a first step towards this global existence conjecture, Andersson, Gudapati and Szeftel proved that the global regularity holds for the equivariant case in \cite{Andersson}, by reference to some pioneering work on equivariant wave maps.

Shatah and Tahvildar-Zadeh have proved the global regularity for 2+1 dimensional equivariant wave maps with the target geodesically convex in \cite{shatah2}. This condition was later relaxed by Grillakis to include a certain class of nonconvex targets, see \cite{grillakis}. Their proof of regularity was also simplified later by Shatah and Struwe in \cite{shatah1}. They gave several more results on equivariant wave maps in the areas of existence and uniqueness, regularity, asymptotic behavior, development of singularities, and weak solutions, see \cite{shatah3}. Further, as an improvement of these above results, for target manifolds that do not admit nonconstant harmonic spheres, global existence of smooth solutions to the Cauchy problem for corotational wave maps with smooth equivariant data was shown by Struwe in \cite{Struwe1}.

 Then, for the 3+1 dimensional vacuum Einstein equations with $G_2$ symmetry, it was shown in \cite{berger} that the system reduce to a spherically symmetric wave map $u:\, \mathbb{R}^{2+1} \to \mathbb{H}^2$, where $\mathbb{R}^{1+2}$ is the 2+1 dimensional Minkowski spacetimes and the target $\mathbb{H}^2$ is the hyperbolic space. Thus the global regularity can be proved by the work of Christodoulou and Tahvildar-Zadeh \cite{christodoulou2} on 2+1 dimensional spherically symmetric wave maps. In \cite{christodoulou2}, the range of the wave map $u$ should be contained in a convex part of the target $N$. This restriction was later shown unnecessary by Struwe in \cite{Struwe3} as the target is the standard sphere. Further, Struwe give a more general result in \cite{Struwe2}, where the target is any smooth, compact Riemannian manifold without boundary. We refer to \cite{geba} for more results and references of wave maps.

\subsection{The 3+1 dimensional spacetime with $U(1) \times \mathbb{R}$ isometry group}

In this paper, we work on the Lorentzian manifold $^{(4)}{M}=\mathbb{R} \times \mathbb{R}^2 \times \mathbb{R}$ with a Lorentzian metric $^{(4)}{g}$ on it, and we consider the polarized case where the Killing fields of $(^{(4)}{M},^{(4)}{g})$ are hypersurface orthogonal. Then, with the existence of the translational Killing vector, the metric can be written in the following form,
\begin{equation*}
^{(4)}{g}=e^{-2\gamma}{^{(3)}g}+e^{2\gamma}{(dx^3)}^2
\end{equation*}
where $\partial_{x^3}$ is the translational spacelike Killing vector field.

As we mentioned before in \cite{choquet3} and \cite{Moncrief1}, the 3+1 dimensional vacuum Einstein equations with spacelike Killing field reduce to a 2+1 dimensional Einstein-wave map system with the target manifold $\mathbb{H}^2$. We gave a similar reduction for \eqref{1.1} to a 2+1 dimensional Einstein-wave-Klein-Gordon system in \cite{chen}, where $\gamma$ in $^{(4)}{g}$ satisfies a wave equation and the scalar field $\phi$ satisfies a Klein-Gordon equation. The equations will be given in section 2.

In the vacuum case with $G_2$ symmetry, the wave maps equations  reduced from the Einstein equations is a semilinear wave equations, for instance, special solutions of this case are the Einstein-Rosen waves, see \cite{cecile3} and references therein. While the major difficulty in our problem is that the wave equations are coupled with Einstein equations, which make the system quasilinear. Moreover, as we can see in the 2+1 dimensional equivariant Einstein-wave map system that Andersson studied in \cite{Andersson}, the coupled unknown $\phi$ which satisfies the wave maps equation vanishes at the axis $\Gamma$. While, there is no such regularity condition for the coupled unknowns $(\gamma,\phi)$ in our case, which brings difficulties in studying the regularity near the first possible singularity on the axis. However, we develop a way to solve this problem in 2+1 dimension with symmetry.

Particularly, if the scalar field is massless, we can remove the condition that the Killing vector field $\partial_{x^3}$ is hypersurface orthogonal, and the metric will take the general form
\begin{equation*}
^{(4)}{g}=e^{-2\gamma}{^{(3)}g}+e^{2\gamma}{(dx^3+A_{\alpha}dx^{\alpha})}^2.
\end{equation*}
We have mentioned in \cite{chen} that the system reduce to a wave map equations coupled with a linear wave equation on the Minkowski spacetimes, of which the problem left is to study the wave map system, same as in the vacuum case.

Now we assume that the reduced spacetime $(^{(3)}M,^{(3)}g)$ is a globally hyperbolic 2+1 dimensional spacetime with Cauchy surface diffeomorphic to $\mathbb{R}^2$, on which the reduced Einstein-wave-Klein-Gordon system is radially symmetric. And we assume that the $U(1)$ action on $M$ is generated by a hypersurface orthogonal Killing field $\partial_{\theta}$. In particular, we write the metric $^{(3)}g$ in the following form in this paper
\begin{eqnarray} \label{1.3}
^{(3)}g&=&-e^{2\alpha(t,r)} dt^2+e^{2\beta(t,r)}dr^2+r^2 d \theta^2\\ \nonumber
& \triangleq &\check{g}+r^2 d \theta^2.
\end{eqnarray}
 where $\check{g}$ is a metric on the orbit space  $\mathcal{Q}=M/ {{\mathbb{S}}^1}$ and $r$ is the radius function, defined such that $2\pi r(p)$ is the length of the ${\mathbb{S}}^1$ orbit through $p$.

 \subsection{The Cauchy data}

As we mentioned before, to study the Cauchy problem of the 3+1 dimensional Einstein-Klein-Gordon system, equivalently in some sense, we can consider the Cauchy problem of the reduced 2+1 dimensional Einstein-wave-Klein-Gordon system.

Now we introduce the definition of the Cauchy data set for the 2+1 dimensional Einstein-wave-Klein-Gordon systems with $U(1)$ isometry group as follows,
\begin{deff}[Cauchy data set for the 2+1 dimensional Einstein-wave-Klein-Gordon system with $U(1)$ isometry group] \label{def1.5}
A Cauchy data set for the 2+1 dimensional Einstein-wave-Klein-Gordon system with a $U(1)$ isometry group is a 7-tuple $(M_0,g_0,K,\gamma_0,\gamma_1,\phi_0,\phi_1)$ consisting of a Remannian 2-manifold $(M_0,g_0)$ with a spacelike rotational Killing vector field and a 2-tensor $K$ which is the second fundamental form and symmetric under the same action, $\gamma_0,\gamma_1$ are initial data for the wave equation that $\gamma$ satisfies, $\phi_0,\phi_1$ are initial data for the Klein-Gordon equation that the scalar field satisfies. $g_0, K$ are functions of $r$ only and the following constraints equations hold:
\begin{equation} \label{1.4}
R_0-K_{\alpha\beta}K^{\alpha\beta}+(trK)^2=2{^{(3)}}T_{\alpha\beta} n^\alpha n^\beta
\end{equation}
\begin{equation} \label{1.5}
D^{\beta} K_{\alpha\beta}-D_{\alpha} K^{\beta}_{\beta}={^{(3)}}{T}_{\alpha\mu} n^{\mu}
\end{equation}
where $n^{\alpha}(t,r)$ is the future directed unit normal, $R_0$ is the scalar curvature on $M_0$,$D_{\alpha}$ is the intrinsic covariant derivative on $M_0$, and $^{(3)}T_{\alpha\beta}$ is the stress-energy tensor for the reduced system.
\end{deff}

For a smooth solution of the 2+1 dimensional Einstein-wave-Klein-Gordon system with $U(1)$ symmetry, it must hold that $\alpha(t,r),\beta(t,r)$ are even functions of $r$. And we give some normalisation of the metric functions $\alpha(t,r)$ and $\beta(t,r)$ on the axis. It must hold that $\beta(t,0)=0$, in order to avoid a conical singularity at the axis $\Gamma$, which means that the perimeter of a circle
of radius $r$ grows like $2\pi cr$ at the axis, instead of $2\pi r$ in the
Euclidean metric. This condition can be realized by appropriately choosing the Cauchy data such that $\beta(0,0)=0$, see \cite{chen}. Further, $\alpha(t,0)$ is determined only up to a choice of time parametrization. We shall choose a time coordinate such that $\alpha(t,0)=0.$

Finding solutions to the constraint equations is a research area in itself. Note that Cecile has proved the existence of such constraint equations in vacuum with translational Killing vector field in \cite{cecile1}\cite{cecile2}, which is used in \cite{cecile3} to prove stability in exponential time of the Minkowski spacetime in this setting. In our case, we just briefly show that such data exist, without going further into the study of the constraints. We have constructed an asymptotically flat\footnote{We say the Cauchy data are 'asymptotically flat' in the sense of Andersson\cite{Andersson} here and hereafter.} Cauchy data set in \cite{chen}, with the energy less than $2\pi$ due to the constraint equations. Meanwhile, the initial data we constructed satisfies the following conditions, under which the solution yield a uniform lower bound for $\gamma$,
\begin{equation} \label{1.19}
\begin{cases}
\gamma_t(0,r)=\gamma_1(r) \geq 0,\\
\gamma(0,r)=\gamma_0(r) \geq 0,\quad\gamma_r(0,r)=\gamma_0'(r)> -\frac{1}{2} r^{-1}.
\end{cases}
\end{equation}

\subsection{The problem of global well-posedness}

The proof by Choquet-Bruhat and Geroch(see \cite{choquet1}\cite{choquet2}) of existence and uniqueness of maximal solutions to the Cauchy problem for the vacuum Einstein equations, together with the equivalence of the Cauchy data, can be generalized to our case as follows, which guaranteed the local well-posedness.
\begin{theorem} \label{thm1.5}
Let $(M_0,g_0,k,\gamma_0,\gamma_1,\phi_0,\phi_1)$ be the Cauchy data set for the 2+1 dimensional Einstein-wave-Klein-Gordon system with $U(1)$ isometry group. Then there is a unique, maximal Cauchy development $(^{(3)}M,^{(3)}g,\gamma,\phi)$ satisfying the the 2+1 dimensional Einstein-wave-Klein-Gordon system.
\end{theorem}

Now we state the main theorem
\begin{theorem} \label{thm1.6}
Let $(^{(3)}M,^{(3)}g)$ be the maximal Cauchy development of a regular Cauchy data set aforementioned in section 1.3 for the 2+1 dimensional Einstein-wave-Klein-Gordon system. Then there is a global in time smooth solution for the Cauchy problem of the equations.
\end{theorem}

\begin{remark} \label{remark1.6}
In our present work above, we have constructed solutions where $(^{(4)}{M},{^{(4)}{g}})$ satisfies the Einstein-scalar field equations equivalently.
\end{remark}

For our further consideration, it is worthwhile to study the asymptotic behaviour of the solution obtained in Theorem \ref{thm1.6}, which is related with the future causal geodesic completeness of the 3+1 dimensional spacetime $(^{(4)}{M},^{(4)}{g})$.

The paper is organized as follows. In Section 2, we reduce the 3+1 dimensional Einstein-Klein-Gordon system with $U(1) \times \mathbb{R}$ isometry group to a 2+1 dimensional Einstein-wave-Klein-Gordon system on $(M,g)$ with $U(1)$ symmetry. In Section 3, we prove that the energy cannot concentrate near the first possible singularity. In Section 4, we give a proof that small energy implying global regularity.

\section{The 2+1 dimensional Einstein-wave-Klein-Gordon system}
Consider the Einstein equations \eqref{1.2}, where the spacetime $(^{(4)}{M},^{(4)}{g})$ admits a spacelike translational Killing vector field. As in \cite{chen}, the Einstein-Klein-Gordon system can be reduced to a 2+1 dimensional Einstein-wave-Klein-Gordon system in a similar way  as is well known. We give the equations in local coordinate system in this section.

\subsection{Equations in $(t,r)$ coordinates system}
When the metric takes the form \eqref{1.3}, we give the computation of the Einstein tensor $G_{\alpha\beta}$ ,
\begin{eqnarray*}
&&G_{00}=\frac{1}{r}e^{2\alpha-2\beta} \beta_r, \\
&&G_{01}=\frac{1}{r} \beta_t, \\
&&G_{11}=\frac{1}{r} \alpha_r, \\
&&G_{22}=r^2 \left( e^{-2\beta} \alpha_{rr}-e^{-2\alpha} \beta_{tt}+e^{-2\beta}\alpha_r (\alpha_r-\beta_r)+e^{-2\alpha} \beta_t (\alpha_t-\beta_t) \right), \\
&&G_{02}=0, \\
&&G_{12}=0.
\end{eqnarray*}
 and the stress energy tensor $T_{\alpha\beta}$,
 \begin{eqnarray*}
&&T_{00}={\gamma_t}^2+e^{2\alpha-2\beta}{\gamma_r}^2+\frac{1}{2}{\phi_t}^2+\frac{1}{2} e^{2\alpha-2\beta} {\phi_r}^2+\frac{m^2}{2}e^{2\alpha-2\gamma}\phi^2 \\
&&=e^{2\alpha} \mathbf{e},\\
&&T_{01}=2\gamma_t \gamma_r+\phi_t\phi_r=e^{(\alpha+\beta)} \mathbf{m}, \\
&&T_{11}=e^{2\beta-2\alpha}{\gamma_t}^2+{\gamma_r}^2+\frac{1}{2}e^{2\beta-2\alpha}{\phi_t}^2+\frac{1}{2} {\phi_r}^2-\frac{m^2}{2}e^{2\beta-2\gamma}\phi^2,\\
&&T_{22}=r^2 e^{-2\alpha}{\gamma_t}^2-r^2 e^{-2\beta}{\gamma_r}^2+\frac{r^2}{2}e^{-2\alpha}{\phi_t}^2-\frac{r^2}{2}e^{-2\beta} {\phi_r}^2-\frac{m^2}{2}r^2 e^{-2\gamma}\phi^2,\\
&&T_{02}=0, \\
&&T_{12}=0.
\end{eqnarray*}

We write the 2+1 dimensional radially symmetric Einstein-wave-Klein-Gordon system in local coordinates,
 \begin{equation} \label{2.10}
{\beta}_r=re^{2\beta-2\alpha}{\gamma_t}^2+r{\gamma_r}^2+\frac{r}{2}e^{2\beta-2\alpha}{\phi_t}^2+\frac{r}{2}{\phi_r}^2+\frac{m^2}{2}re^{2\beta-2\gamma} \phi^2
\end{equation}
\begin{equation} \label{2.11}
\beta_t=2r\gamma_t \gamma_r+r \phi_t \phi_r
\end{equation}
\begin{equation} \label{2.12}
\alpha_r=re^{2\beta-2\alpha}{\gamma_t}^2+r{\gamma_r}^2+\frac{r}{2}e^{2\beta-2\alpha}{\phi_t}^2+\frac{r}{2}{\phi_r}^2-\frac{m^2}{2}re^{2\beta-2\gamma} \phi^2
\end{equation}
\begin{eqnarray} \label{2.13}
 &&e^{-2\beta} \alpha_{rr}-e^{-2\alpha} \beta_{tt}+e^{-2\beta} \alpha_r(\alpha_r-\beta_r)+e^{-2\alpha} \beta_t (\alpha_t-\beta_t) \\ \nonumber
 &&=-\frac{m^2}{2}e^{-2\gamma} \phi^2+ e^{-2\alpha}({\gamma_t}^2+\frac{1}{2}{\phi_t}^2)-e^{-2\beta}({\gamma_r}^2+\frac{1}{2}{\phi_r}^2)
 \end{eqnarray}
\begin{eqnarray} \label{2.14}
\Box_{^{(3)}g}\gamma &=& -e^{-2\alpha}(\gamma_{tt}+(\beta_t-\alpha_t) \gamma_t)+e^{-2\beta} (\gamma_{rr}+\frac{\gamma_r}{r}+(\alpha_r-\beta_r) \gamma_r) \\ \nonumber
&=& -\frac{m^2}{2}e^{-2\gamma}\phi^2
\end{eqnarray}
\begin{eqnarray} \label{2.15}
\Box_{^{(3)}g} \phi &=& -e^{-2\alpha}(\phi_{tt}+(\beta_t-\alpha_t) \phi_t)+e^{-2\beta} (\phi_{rr}+\frac{\phi_r}{r}+(\alpha_r-\beta_r) \phi_r) \\ \nonumber
&=& m^2e^{-2\gamma} \phi.
\end{eqnarray}

\subsection{Null coordinates}
In this section, we write the equations in a null coordinate system introduced in \cite{chen}, in which the wave equations may be written in a classical form in the flat case. In the following part, we assume that all objects are smooth, unless otherwise stated.

 Let $(\mathcal{Q},\check{g})$ be the orbit space, where
 \begin{displaymath}
 \mathcal{Q}=M/{{\mathbb{S}}^1}
 \end{displaymath}
 and
 \begin{equation*}
 \check{g}=-e^{2\alpha} dt^2+e^{2\beta} dr^2.
 \end{equation*}

In \cite{chen}, we constructed a null coordinate system with respect to which $\check{g}$ takes the form
 \begin{displaymath}
 \check{g}=-e^{2\lambda}(u,v)dudv
 \end{displaymath}
 which means that the 3-dimensional manifold $(M,g)$ admits a coordinate system $(u,v,\theta)$ such that $g$ takes the form
  \begin{displaymath}
g=-e^{2\lambda}(u,v)dudv+r^2(u,v)d\theta^2
 \end{displaymath}
 where now $d\theta^2$ is the line element on the $\mathbb{S}^1$ symmetry orbit. Meanwhile, the null coordinate system satisfies initial boundary conditions,
\begin{equation} \label{4.6.1}
 \begin{cases}
t=0: u=-v\\
r=0: u=t,\,v=t
 \end{cases}
 \end{equation}

In null coordinates, the components of the Einstein tensor take the following form
\begin{eqnarray*}
&&G_{00}=-e^{2\lambda}r^{-1} \partial_u(e^{-2\lambda}\partial_u r), \\
&&G_{01}=r^{-1}\partial_u\partial_v r, \\
&&G_{11}=-e^{2\lambda}r^{-1} \partial_v(e^{-2\lambda}\partial_v r), \\
&&G_{22}=-4r^2 e^{-2\lambda}\partial_u\partial_v \lambda.
\end{eqnarray*}
Other components are zero.

Thus, rewritting the Einstein-wave-Klein-Gordon system \eqref{2.10}-\eqref{2.15} in null coordinates, we can get
\begin{equation} \label{1.12}
\partial_u(e^{-2\lambda} \partial_u r)=-e^{-2\lambda} r(2{\gamma_u}^2+{\phi_u}^2)
\end{equation}
\begin{equation} \label{1.13}
r_{uv}=\frac{m^2}{4}r e^{2\lambda-2\gamma} {\phi}^2
\end{equation}
\begin{equation} \label{1.14}
\partial_v(e^{-2\lambda} \partial_v r)=-e^{-2\lambda} r(2{\gamma_v}^2+{\phi_v}^2)
\end{equation}
\begin{equation} \label{1.15}
\lambda_{uv}=-\gamma_u \gamma_v-\frac{1}{2} \phi_u \phi_v+\frac{m^2}{8} e^{2\lambda-2\gamma} {\phi}^2
\end{equation}
\begin{equation} \label{1.16}
\partial_u (r \partial_v \gamma)+\partial_v (r \partial_u \gamma)=\frac{m^2}{4}re^{2\lambda-2\gamma} {\phi}^2
\end{equation}
\begin{equation} \label{1.17}
\partial_u (r \partial_v \phi)+\partial_v (r \partial_u \phi)=-\frac{m^2}{2}r e^{2\lambda-2\gamma} \phi
\end{equation}
with stress-energy tensor,
\begin{eqnarray*}
&&T_{00}=2{\gamma_u}^2 +{\phi_u}^2, \\
&&T_{01}=\frac{m^2}{4} e^{2\lambda-2\gamma} {\phi}^2, \\
&&T_{11}=2{\gamma_v}^2 +{\phi_v}^2, \\
&&T_{22}=4r^2 e^{-2\lambda}\gamma_u \gamma_v+2r^2e^{-2\lambda} \phi_u \phi_v-r^2 e^{-2\gamma}\frac{m^2}{2}{\phi}^2.
\end{eqnarray*}

 Then, let us define
 \begin{displaymath}
 T=\frac{u+v}{2},\, R=\frac{v-u}{2}.
 \end{displaymath}

We can also rewrite the system in $(T,R)$ coordinates which reads
\begin{equation} \label{1.6}
\frac{r_T}{r} \lambda_T+\frac{r_R}{r} \lambda_R-\frac{r_{RR}}{r}={\gamma_T}^2+{\gamma_R}^2+\frac{1}{2}{\phi_T}^2+\frac{1}{2}{\phi_R}^2+e^{2\lambda-2\gamma} \frac{m^2}{2}{\phi}^2
\end{equation}
\begin{equation} \label{1.7}
-\frac{r_{TR}}{r}+\frac{r_T}{r} \lambda_R+\frac{r_R}{r} \lambda_T=2\gamma_T \gamma_R+\phi_T \phi_R
\end{equation}
\begin{equation} \label{1.8}
\frac{r_T}{r} \lambda_T+\frac{r_R}{r} \lambda_R-\frac{r_{TT}}{r}={\gamma_T}^2+{\gamma_R}^2+\frac{1}{2}{\phi_T}^2+\frac{1}{2}{\phi_R}^2-e^{2\lambda-2\gamma} \frac{m^2}{2}{\phi}^2
\end{equation}
\begin{equation} \label{1.9}
e^{-2\lambda} \lambda_{RR}-e^{-2\lambda} \lambda_{TT} =-e^{-2\gamma} \frac{m^2}{2}{\phi}^2+e^{-2\lambda}{\gamma_T}^2-e^{-2\lambda}{\gamma_R}^2+\frac{1}{2} e^{-2\lambda} {\phi_T}^2-\frac{1}{2} e^{-2\lambda} {\phi_R}^2
\end{equation}
\begin{eqnarray} \label{1.10}
\Box_{g}\gamma &=& -e^{-2\lambda}\gamma_{TT}+e^{-2\lambda} \gamma_{RR}-e^{-2\lambda}\frac{r_T}{r} \gamma_T+e^{-2\lambda} \frac{r_R}{r} \gamma_R \\ \nonumber
&=& -e^{-2\gamma}\frac{m^2}{2}{\phi}^2
\end{eqnarray}
\begin{eqnarray} \label{1.11}
\Box_{g} \phi &=& -e^{-2\lambda}\phi_{TT}+e^{-2\lambda} \phi_{RR}-e^{-2\lambda}\frac{r_T}{r} \phi_T+e^{-2\lambda} \frac{r_R}{r} \phi_R\\ \nonumber
&=& e^{-2\gamma} m^2 \phi
\end{eqnarray}
with the stress-energy tensor $T_{\mu\nu}$,
\begin{eqnarray*}
&&T_{00}={\gamma_T}^2+{\gamma_R}^2+\frac{1}{2}{\phi_T}^2+\frac{1}{2}  {\phi_R}^2+e^{2\lambda-2\gamma}\frac{m^2}{2}{\phi}^2=e^{2\lambda} \tilde{\mathbf{e}}, \\
&&T_{01}=2\gamma_T \gamma_R+ \phi_T \phi_R, \\
&&T_{11}={\gamma_T}^2+{\gamma_R}^2+\frac{1}{2}{\phi_T}^2+\frac{1}{2} {\phi_R}^2-e^{2\lambda-2\gamma}\frac{m^2}{2}{\phi}^2, \\
&&T_{22}=r^2 e^{-2\lambda}{\gamma_T}^2-r^2 e^{-2\lambda}{\gamma_R}^2+\frac{r^2}{2}e^{-2\lambda}{\phi_T}^2-\frac{r^2}{2}e^{-2\lambda} {\phi_R}^2-r^2 e^{-2\gamma}\frac{m^2}{2}{\phi}^2.
\end{eqnarray*}

Let $U=(\gamma,\phi)$, rewrite equations $(\ref{1.10})(\ref{1.11})$ in the following form on the Minkowski spacetimes $(\mathbb{R}^{1+2},m)$,
\begin{eqnarray} \label{1.18}
\Box_{m} U &=& -U_{TT}+U_{RR}+ \frac{1}{R} U_R\\ \nonumber
&=& \mathcal{H}+\frac{r_T}{r} U_T-\left(\frac{r_R}{r}-\frac{1}{R}\right) U_R \\ \nonumber
& \triangleq & h
\end{eqnarray}
where $\mathcal{H}=\mathcal{H}_1 \triangleq -e^{2\lambda-2\gamma} \frac{m^2}{2} \phi^2$ when we consider the wave equation \eqref{1.10} of $\gamma$, and similarly for \eqref{1.11} we have $\mathcal{H}=\mathcal{H}_2 \triangleq e^{2\lambda-2\gamma} m^2 \phi.$

\section{Non-concentration of energy}
 As a preliminary part, the energy estimates of the Einstein-wave-Klein-Gordon system have been given in \cite{chen}. We have shown the conservation and the monotonicity of the energy. Using the energy estimates, we proved that the metric functions $\alpha,\beta$ and $\lambda$ are  uniformly bounded. We will recall some notations in this section. For simplicity, we denote $(^{(3)}M,^{(3)}g)$ by $(M,g)$ from now on.

Let us define the energy on the Cauchy surface $\Sigma_t$,
\begin{eqnarray*}
E(t)&:=&\int_{\Sigma_t} \mathbf{e} {\bar{\mu}}_q \\
&=& 2\pi \int_0^{\infty} \mathbf{e}(t,r) r e^{\beta(t,r)} dr,
\end{eqnarray*}
the energy in a coordinate ball $B_r$,
\begin{eqnarray*}
E(t,r)& :=& \int_{B_r} \mathbf{e} {\bar{\mu}}_q \\
&=&2\pi \int_0^r \mathbf{e}(t,r') r' e^{\beta(t,r')} d r',
\end{eqnarray*}
the energy inside the causal past $J^{-}(O)$ of $O$,
\begin{displaymath}
E^O(t):=\int_{{\Sigma_t} \cap {J^-(O)}} \mathbf{e} {\bar{\mu}}_q
\end{displaymath}
with $O$ the first possible singularity. In this section, by shifting time we may assume that $O$ is the origin.

We aim to show the following theorem in this section.
\begin{theorem} \label{thm3.1}
Let $(M,g)$ be the maximal Cauchy development of the Cauchy data set in Theorem \ref{thm1.6}. Then, the energy of the Einstein-wave-Klein-Gordon system \eqref{2.10}-\eqref{2.15} cannot concentrate, i.e. $E^O(t)\to 0$ as $t \to 0$, where $O$ is the first possible singularity.
\end{theorem}
The proof of the non-concentration of energy will be performed in a similar scheme as Andersson did in \cite{Andersson}.

\subsection{The Vector field Method}
Let $X$ be a vector field on $M$. Set the corresponding momentum $P_X$ as follows
\begin{equation} \label{3.7}
P_X^{\mu}=T^{\mu}_{\nu} X^{\nu},
\end{equation}
then, we have
\begin{equation} \label{3.8}
\nabla_{\nu} P_X^{\nu}=X^{\mu} \nabla_{\nu} T^{\nu}_{\mu}+T^{\nu}_{\mu} \nabla_{\nu} X^{\mu}.
\end{equation}
Since the stress-energy tensor $T_{\mu\nu}$ satisfies
\begin{displaymath}
\nabla^{\mu} T_{\mu\nu}=0,
\end{displaymath}
the first term in the right hand side vanishes, hence
\begin{eqnarray*}
\nabla_{\nu} P_X^{\nu}&=&T^{\mu\nu} \nabla_{\mu} X_{\nu} \\
&=& \frac{1}{2}{ ^{(X)}{\pi}_{\mu\nu} T^{\mu\nu}},
\end{eqnarray*}
where the deformation tensor $ ^{(X)} {\pi}_{\mu\nu}$ is defined by
\begin{eqnarray*}
 ^{(X)} {\pi}_{\mu\nu} &:=& \nabla_{\mu} X_{\nu}+\nabla_{\nu} X_{\mu} \\
 &=& g_{\sigma\nu} \partial_{\mu} X^{\sigma}+g_{\sigma\mu} \partial_{\nu} X^{\sigma}+X^{\sigma} \partial_{\sigma} g_{\mu\nu}.
\end{eqnarray*}

In the following let us calculate the divergence of $P_X$ for various choices of $X$. Consider $T=e^{-\alpha}\partial_t$, the corresponding momentum $P_T$ is
\begin{equation} \label{3.9}
P_T=-e^{-\alpha} \mathbf{e} \partial_t+e^{-\beta} \mathbf{m} \partial_r.
\end{equation}
We have shown that $P_T$ is divergence free,
\begin{eqnarray} \label{3.10}
\nabla_{\nu} P_T^{\nu}&=&e^{-\alpha} \beta_t (\mathbf{e}-\mathbf{f})-e^{-\beta}\alpha_r \mathbf{m} \\ \nonumber
&=& 0,
\end{eqnarray}
where
\begin{displaymath}
\mathbf{f}:=m^2 e^{-2\gamma} \phi^2.
\end{displaymath}

Equivalently,
\begin{eqnarray} \label{3.4}
\nabla_{\nu} P_T^{\nu}&=&\frac{1}{\sqrt{|g|}} \partial_{\nu} \left( \sqrt{|g|} P_T^{\nu} \right) \\ \nonumber
&=& \frac{1}{r e^{\beta+\alpha}} \left(-\partial_t (r e^{\beta} \mathbf{e})+\partial_r(re^{\alpha} \mathbf{m})\right)\\ \nonumber
&=& 0.
\end{eqnarray}

For $R=e^{-\beta}\partial_r$, the corresponding momentum $P_R$ is
\begin{equation} \label{3.11}
P_R=-e^{-\alpha} \mathbf{m} \partial_t+e^{-\beta} (\mathbf{e}-\mathbf{f}) \partial_r.
\end{equation}
the divergence of $P_R$ is
\begin{eqnarray} \label{3.12}
\nonumber
\nabla_{\nu} P_R^{\nu}&=& \frac{1}{2}{ ^{(R)}{\pi}_{\mu\nu} T^{\mu\nu}}\\
&=& -e^{-\beta} \alpha_r \mathbf{e}+e^{-\alpha}\beta_t \mathbf{m}+\frac{1}{2r}e^{-\beta}(2e^{-2\alpha} {\gamma_t}^2+e^{-2\alpha}{\phi_t}^2-2e^{-2\beta} {\gamma_r}^2-e^{-2\beta}{\phi_r}^2-\mathbf{f}).
\end{eqnarray}
which is also given by
\begin{eqnarray} \label{3.13}
\nonumber
\nabla_{\nu} P_R^{\nu}&=& \frac{1}{\sqrt{-g}}\partial_{\nu} (\sqrt{-g}P_R^{\nu})\\
&=&\frac{1}{re^{\beta+\alpha}}(-\partial_t (re^{\beta}\mathbf{m})+\partial_r (re^{\alpha}(\mathbf{e}-\mathbf{f}))).
\end{eqnarray}

Similarly for the choice $\mathcal{R}_1=r \partial_r$, we have
\begin{equation*}
P_{\mathcal{R}_1}=-e^{\beta-\alpha}r \mathbf{m} \partial_t+r (\mathbf{e}-\mathbf{f}) \partial_r.
\end{equation*}
Using the Einstein equations \eqref{2.10} and \eqref{2.12}, we have
\begin{equation} \label{3.14}
\nabla_{\nu} P_{\mathcal{R}_1}^{\nu}= 2e^{-2\alpha} {\gamma_t}^2+e^{-2\alpha}{\phi_t}^2-m^2 e^{-2\gamma} \phi^2.
\end{equation}

Now let $J^-(O)$ be the causal past of the point $O$ and $I^-(O)$ the chronological past of $O$. Compared to the flat case, we give the following definitions
\begin{displaymath}
\Sigma_t^O:=\Sigma_t \cap J^-(O),
\end{displaymath}
\begin{displaymath}
K(t):=\cup_{t_0 \leq t \leq t' <0} \Sigma_{t'} \cap J^-(O),
\end{displaymath}
\begin{displaymath}
C(t):=\cup_{t_0 \leq t \leq t' < 0} \Sigma_{t'} \cap \left(J^-(O) \setminus I^-(O) \right),
\end{displaymath}
\begin{displaymath}
K(t,s):=\cup_{t_0 \leq t \leq t' < s} \Sigma_{t'} \cap J^-(O),
\end{displaymath}
\begin{displaymath}
C(t,s):=\cup_{t_0 \leq t \leq t' < s} \Sigma_{t'} \cap \left(J^-(O) \setminus I^-(O) \right)
\end{displaymath}
for $t_0 \leq t<s<0$ with $t_0$ the initial time. In the following we will try to understand the behaviour of various quantities of the system as one approaches $O$ in a limiting sense. For this purpose we will use Stokes' theorem in the region $K(\tau,s)$, $t_0 \leq \tau \leq s<0$.

The volume 3-form of $(M,g)$ is given by
\begin{displaymath}
{\bar{\mu}}_g=r e^{\beta+\alpha} dt \wedge dr \wedge d \theta
\end{displaymath}
and the area 2-form of $(\Sigma,q)$ by
\begin{displaymath}
{\bar{\mu}}_q=r e^{\beta} dr \wedge d \theta.
\end{displaymath}
Let us define 1-forms $\tilde{l},\tilde{n}$ and $\tilde{m}$ as follows
\begin{displaymath}
\tilde{l}:=-e^{\alpha}dt+e^{\beta}dr,
\end{displaymath}
\begin{displaymath}
\tilde{n}:=-e^{\alpha}dt-e^{\beta}dr,
\end{displaymath}
\begin{displaymath}
\tilde{m}:=r d \theta,
\end{displaymath}
therefore,
\begin{displaymath}
{\bar{\mu}}_g=\frac{1}{2} \left( \tilde{l} \wedge \tilde{n} \wedge \tilde{m} \right).
\end{displaymath}
Then we introduce the 2-forms ${\bar{\mu}}_{\tilde{l}}$ and ${\bar{\mu}}_{\tilde{n}}$ such that
\begin{displaymath}
{\bar{\mu}}_{\tilde{l}}:=-\frac{1}{2} \tilde{n} \wedge \tilde{m},
\end{displaymath}
\begin{displaymath}
{\bar{\mu}}_{\tilde{n}}:=\frac{1}{2} \tilde{l} \wedge \tilde{m},
\end{displaymath}
so we have
\begin{displaymath}
{\bar{\mu}}_g=-\tilde{l} \wedge {\bar{\mu}}_{\tilde{l}},
\end{displaymath}
\begin{displaymath}
{\bar{\mu}}_g=-\tilde{n} \wedge {\bar{\mu}}_{\tilde{n}}.
\end{displaymath}
Now, let us apply the Stokes' theorem for the ${\bar{\mu}}_g$-divergence of $P_X$ in the region $K(\tau,s)$. We have
\begin{equation} \label{3.16}
\int_{K(\tau,s)} \nabla_{\nu} P_X^{\nu} {\bar{\mu}}_g= \int_{\Sigma_s^O} e^{\alpha} P_X^t {\bar{\mu}}_q-\int_{\Sigma_{\tau}^O} e^{\alpha}P_X^t {\bar{\mu}}_q+Flux(P_X)(\tau,s)
\end{equation}
where
\begin{displaymath}
Flux(P_X)(\tau,s)=-\int_{C(\tau,s)} \tilde{n}(P_X) {\bar{\mu}}_{\tilde{n}}.
\end{displaymath}

On the other hand, in null coordinate system, see \cite{chen} section 4, we have
\begin{displaymath}
d v=-e^{-\mathcal{F}} \tilde{n},\, du=-e^{-\mathcal{G}} \tilde{l}.
\end{displaymath}
And the volume 3-form of $(M,g)$ takes the form
\begin{displaymath}
{\bar{\mu}}_g=\frac{1}{2} r e^{2\lambda} du \wedge dv \wedge d \theta.
\end{displaymath}
Next, we introduce the 2-forms ${\bar{\mu}}_v$ and ${\bar{\mu}}_u$ as follows
\begin{displaymath}
{\bar{\mu}}_g=dv \wedge {\bar{\mu}}_v,\, {\bar{\mu}}_g=du \wedge {\bar{\mu}}_u.
\end{displaymath}
From the above two formulas, we infer
\begin{displaymath}
{\bar{\mu}}_v=-\frac{1}{2} r e^{2\lambda} ( du \wedge d \theta),\, {\bar{\mu}}_u=\frac{1}{2}r e^{2\lambda} (dv \wedge d \theta).
\end{displaymath}
Therefore,
\begin{displaymath}
Flux(P_X)(\tau,s)=\int_{C(\tau,s)} dv(P_X) {\bar{\mu}}_v,
\end{displaymath}
in particular,
\begin{eqnarray*}
Flux(P_T)(\tau,s)&=&\int_{C(\tau,s)} dv(P_T) {\bar{\mu}}_v, \\
&=&-\int_{C(\tau,s)} e^{-\mathcal{F}} (\mathbf{e}-\mathbf{m}) {\bar{\mu}}_v.
\end{eqnarray*}

\subsection{Monotonicity of Energy}
As was mentioned in \cite{chen}, we have shown that the monotonicity of energy holds.
\begin{proposition} \label{prop3.3}
We have $E^O(\tau) \geq E^O(s)$ for $t_0 \leq \tau<s<0$.
\end{proposition}

Now, we define
\begin{equation} \label{3.17}
E^O_{conc}:=\inf\limits_{\tau \in [t_0,0)} E^O(\tau).
\end{equation}
By Proposition \ref{prop3.3}, \eqref{3.17} is equivalent to
\begin{equation} \label{3.18}
E^O_{conc}:=\lim\limits_{\tau \to 0} E^O(\tau).
\end{equation}
We say that the energy of the Cauchy problem concentrates if $E^O_{conc} \neq 0$ and does not concentrate if $E^O_{conc}=0$. By the monotonicity above, we immediately have the following proposition
\begin{proposition} \label{prop3.4}
For the vector field $T$, let
\begin{displaymath}
Flux(P_T)(\tau):=\lim\limits_{s \to 0} Flux(P_T)(\tau,s).
\end{displaymath}
Then, we have $Flux(P_T)(\tau) \to 0$ as $\tau \to 0$.
\end{proposition}

\subsection{Non-concentration of integrated potential energy}
In the following proposition, we shall prove that the potential energy does not concentrate.
\begin{proposition} \label{prop3.5}
The following potential energy does not concentrate, i.e.
\begin{equation} \label{3.19}
\int_{\Sigma_{\tau}\cap J^-(O)} \mathbf{f}  \bar{\mu}_q \to 0 \quad as \quad \tau \to 0.
\end{equation}
\end{proposition}
\begin{proof}
By H\"{o}lder's inequality and the fact that $\gamma$ has a lower bound, we have
\begin{eqnarray*}
\int_{\Sigma_{\tau}\cap J^-(O)} \mathbf{f}  \bar{\mu}_q
&\lesssim& {\left(\int_{\Sigma_{\tau}} \phi^4  \bar{\mu}_q\right)}^{\frac{1}{2}} {  r_2(\tau)}.
\end{eqnarray*}
where $r=r_2(\tau)$ is the radius where the $t=\tau$ slice intersects the $R=|T|$ curve, i.e the mantel of the null cone $J^-(O)$.

For $U=(\gamma,\phi)$, using Sobolev inequality, we have
\begin{eqnarray*}
{(\int_{\Sigma_{\tau}}U^4 \bar{\mu}_q)}^{\frac{1}{4}} &\lesssim& {(\int_{\Sigma_{\tau}} U^2  rdr)}^{\frac{1}{2}}+{(\int_{\Sigma_{\tau}} {U_r}^2  rdr)}^{\frac{1}{2}}\\
&\lesssim& {(\int_{\Sigma_{\tau}} U^2  rdr)}^{\frac{1}{2}}+E_0\\
&\lesssim& {\| U(\tau)  \|}_{L^2}+E_0.
\end{eqnarray*}
By triangular inequality and Minkowski inequality, we can get
\begin{eqnarray} \label{3.20}
{\| U(\tau)  \|}_{L^2} &\lesssim& {\| U(t_0)  \|}_{L^2}+{\|\int_{t_0}^{\tau} U_t dt\|}_{L^2}\\ \nonumber
&\lesssim& C+\int_{t_0}^{\tau} {\|U_t \|}_{L^2} dt\\ \nonumber
&\lesssim& C.
\end{eqnarray}
where $t_0$ denotes the initial time.

Thus,
\begin{equation} \label{3.21}
{(\int_{\Sigma_{\tau}} U^4  \bar{\mu}_q)}^{\frac{1}{4}} \lesssim C.
\end{equation}
Then, by \eqref{3.21}, we imply
\begin{equation*}
\int_{\Sigma_{\tau}\cap J^-(O)} \mathbf{f}  \bar{\mu}_q \lesssim  {r_2(\tau)}.
\end{equation*}
with $r_2(\tau) \to 0$ as $\tau \to 0$. This concludes the proof of the proposition.
\end{proof}

\subsection{Non-concentration away from the axis}
In this section, we shall prove that energy does not concentrate away from the axis using the divergence free vector $P_T$.
\begin{proposition} \label{prop3.6}
The following energy on an annular slice away from the axis does not concentrate,
\begin{equation*}
E_{ext}^O (\tau):=\int_{B_{r_2(\tau)} \setminus B_{r_1(\tau)}} \mathbf{e} \bar{\mu}_q \to 0 \quad as \quad \tau \to 0,
\end{equation*}
where $r=r_2(\tau)$ is the radius where the $t=\tau$ slice intersects the $R=|T|$ curve i.e. the mantel of the null cone $J^{-}(O)$ and $r=r_1(\tau)$ is the radius where the $t=\tau$ slice intersects the $R=\lambda|T|$ curve, for any real $\lambda \in (0,1)$. Observe that both $r_1(\tau)$ and $r_2(\tau) \to 0$ as $\tau \to 0$.
\end{proposition}
\begin{proof}
Consider a tubular region $\mathcal{S}$ with triangular cross section as in \cite{Andersson} in $R>\lambda |T|, \lambda \in (0,1)$ of the spacetime, i.e. the "exterior" part of the interior of the past null cone of $O$. With the use of the divergence-free vector field $P_T$ and Stokes' theorem in a triangular region $\mathcal{S}$ with three boundary segments $\partial \mathcal{S}_1,\partial \mathcal{S}_2$ and $\partial \mathcal{S}_3$, we can estimate the "exterior" energy by calculating the fluxes instead. Here, $\partial \mathcal{S}_3=\overline{B_{r_2(\tau)} \setminus B_{r_1(\tau)}}$, $\partial \mathcal{S}_2$ is a section of the mantel of the null cone $J^{-}(O)$ , $\partial \mathcal{S}_1$ is the outgoing null surface issuing from $(\tau,r_1(\tau))$ and intersects with $R=|T|$. Thus, we obtain
\begin{eqnarray} \label{3.22}
0 &=& \int_{\partial \mathcal{S}_1} du(P_T) \bar{\mu}_u+ \int_{\partial \mathcal{S}_2} dv(P_T) \bar{\mu}_v-\int_{\partial \mathcal{S}_3} e^{\alpha}P^t_T \bar{\mu}_q \\ \nonumber
&=&-\int_{\partial \mathcal{S}_1} e^{-\mathcal{G}}(\mathbf{e}+\mathbf{m}) \bar{\mu}_u-\int_{\partial \mathcal{S}_2} e^{-\mathcal{F}}(\mathbf{e}-\mathbf{m}) \bar{\mu}_v+\int_{\partial \mathcal{S}_3} \mathbf{e} \bar{\mu}_q.
\end{eqnarray}
The first flux term $\int_{\partial \mathcal{S}_2}$ tends to $0$ when approaching $O$ according to Proposition \ref{prop3.4}. To analyze the behaviour of another flux term $\int_{\partial \mathcal{S}_1}$ in \eqref{3.22} close to $O$, we give the following quantities as in \cite{Andersson},
\begin{displaymath}
\hat{l}:=e^{\beta+\alpha}\tilde{l}=e^{\beta}\partial_t+e^{\alpha}\partial_r,
\end{displaymath}
\begin{displaymath}
\hat{n}:=e^{\beta+\alpha}\tilde{n}=e^{\beta}\partial_t-e^{\alpha}\partial_r,
\end{displaymath}
\begin{displaymath}
\mathcal{F}^2:=r(\mathbf{e}-\mathbf{m}),
\end{displaymath}
\begin{displaymath}
\mathcal{G}^2:=r(\mathbf{e}+\mathbf{m}).
\end{displaymath}
From \eqref{3.12} and \eqref{3.13}, we have
\begin{eqnarray} \label{3.23}
&&\frac{1}{re^{\beta+\alpha}}(-\partial_t (re^{\beta}\mathbf{m})+\partial_r (re^{\alpha}(\mathbf{e}-\mathbf{f}))) \\ \nonumber
&&=-e^{-\beta} \alpha_r \mathbf{e}+e^{-\alpha}\beta_t \mathbf{m}+\frac{1}{2r}e^{-\beta}(2e^{-2\alpha} {\gamma_t}^2+e^{-2\alpha}{\phi_t}^2-2e^{-2\beta} {\gamma_r}^2-e^{-2\beta}{\phi_r}^2-\mathbf{f}).
\end{eqnarray}
Then, using \eqref{3.4} and \eqref{3.23}, we can get
\begin{equation} \label{3.24}
\partial_t (r e^{\beta} \mathbf{e})-\partial_r(re^{\alpha} \mathbf{m})=0,
\end{equation}
\begin{equation} \label{3.25}
\partial_t (r e^{\beta} \mathbf{m})-\partial_r(re^{\alpha} \mathbf{e})=\mathcal{L},
\end{equation}
where
\begin{displaymath}
\mathcal{L}:=\frac{re^{\alpha} \alpha_r}{2} \left(2{(T\gamma)}^2+{(T\phi)}^2+2{(R\gamma)}^2+{(R\phi)}^2-\mathbf{f}\right)+e^{\alpha}{\mathcal{L}}_0-r\beta_te^{\beta}\mathbf{m}
\end{displaymath}
for
\begin{eqnarray*}
{\mathcal{L}}_0&\triangleq&\frac{1}{2} \left(-2{(T\gamma)}^2-{(T\phi)}^2+2{(R\gamma)}^2+{(R\phi)}^2+\mathbf{f}\right)-r\partial_r \mathbf{f}-\mathbf{f} \\
&=&\frac{1}{2} \left(-2{(T\gamma)}^2-{(T\phi)}^2+2{(R\gamma)}^2+{(R\phi)}^2\right)-2m^2e^{-2\gamma}r\left(\phi\phi_r-\phi^2 \gamma_r\right)-\frac{1}{2}m^2 e^{-2\gamma} \phi^2.
\end{eqnarray*}
Furthermore, by adding and subtracting the identities \eqref{3.24} and \eqref{3.25}, we can get
\begin{equation} \label{3.26}
\partial_\nu (re^{\beta+\alpha}(\mathbf{e}-\mathbf{m})\tilde{l}^\nu)=\partial_\nu (\mathcal{F}^2\hat{l}^\nu)=-\mathcal{L},
\end{equation}
\begin{equation} \label{3.27}
\partial_\nu (re^{\beta+\alpha}(\mathbf{e}+\mathbf{m})\tilde{n}^\nu)=\partial_\nu (\mathcal{G}^2\hat{n}^\nu)=\mathcal{L}.
\end{equation}
As in \cite{Andersson}, we express $\mathcal{L}$ in terms of $\mathcal{F}^2\mathcal{G}^2$ by using the Einstein equations,
\begin{equation} \label{3.28}
\mathcal{L}=e^\alpha {\mathcal{L}}_0+e^{2\beta+\alpha}\left(\mathcal{F}^2\mathcal{G}^2-2r^2 \mathbf{e}\mathbf{f}+r^2\mathbf{f}^2 \right).
\end{equation}
In what follows we use Einstein equations to estimate the nonlinear terms involving $\mathbf{e}$ and $\mathbf{f}$ in the quantity $\mathcal{L}$, for the purpose of setting up a Gr\"{o}nwall estimate for $\mathcal{F}$ and $\mathcal{G}$ using the identities in \eqref{3.26}\eqref{3.27}.

Firstly, as shown in \cite{Andersson}, we have that
\begin{displaymath}
\hat{l}^{\mu} \partial_{\mu} e^{2\beta}=2e^{2\beta} e^{2\beta+\alpha} \mathcal{G}^2, \quad \hat{n}^{\mu} \partial_{\mu} e^{2\beta}=-2e^{2\beta} e^{2\beta+\alpha} \mathcal{F}^2,
\end{displaymath}
and
\begin{displaymath}
\partial_{\mu} \hat{l}^{\mu}=e^{2\beta+\alpha}(\mathcal{G}^2-r\mathbf{f}), \quad \partial_{\mu} \hat{n}^{\mu}=e^{2\beta+\alpha}(-\mathcal{F}^2+r\mathbf{f}).
\end{displaymath}
These imply
\begin{displaymath}
\hat{l}^{\mu} \partial_{\mu} (e^{2\beta} \mathcal{F}^2)=e^{2\beta+\alpha}(-{\mathcal{L}}_0+r^2e^{2\beta} S_1),
\end{displaymath}
where
\begin{eqnarray*}
S_1 &=& 3 \mathbf{e}\mathbf{f}-\mathbf{f}^2-\mathbf{m}\mathbf{f} \\
&=&  (\mathbf{e}-\mathbf{m})\mathbf{f}+\mathbf{e_0}\mathbf{f} \\
&\geq& 0,
\end{eqnarray*}
and
\begin{displaymath}
\hat{n}^{\mu} \partial_{\mu} (e^{2\beta} \mathcal{G}^2)=e^{2\beta+\alpha}({\mathcal{L}}_0+r^2e^{2\beta} S_2),
\end{displaymath}
where
\begin{eqnarray*}
S_2 &=& -3 \mathbf{e}\mathbf{f}+\mathbf{f}^2-\mathbf{m}\mathbf{f} \\
&=& - (\mathbf{e}+\mathbf{m})\mathbf{f}-\mathbf{e_0}\mathbf{f} \\
&\leq& 0.
\end{eqnarray*}
Now we define the quantities $\hat{\mathcal{F}}$ and  $\hat{\mathcal{G}}$ in the following way
\begin{displaymath}
\hat{\mathcal{F}}:=e^{\beta}\mathcal{F}, \quad \hat{\mathcal{G}}:=e^{\beta}\mathcal{G}.
\end{displaymath}
In what follows, we give an estimate on ${{\mathcal{L}}_0}^2$ controlled by $\mathbf{e}^2-\mathbf{m}^2$. First, we have that,
\begin{equation*}
{{\mathcal{L}}_0}^2 \lesssim r^2 \mathbf{f}^2 \mathbf{e_0}+r^2 \mathbf{e_0} \mathbf{f}+\mathbf{f}^2+\frac{1}{4} \mathbf{e_0}^2-\mathbf{m}^2
\end{equation*}
We claim that there exists a constant $C>0$ such that
\begin{equation} \label{3.29}
r{|U|}^2 \leq C.
\end{equation}
This is because
\begin{eqnarray*}
r{|U|}^2 &\leq& r \left(-\int_r^{+\infty} \partial_\rho U^2 d \rho \right) \\
&\leq& r \left(-2\int_r^{+\infty} U \times U_\rho d \rho \right) \\
&\lesssim&  \left(\int_r^{+\infty} \rho |U| \times |U_\rho| d \rho \right) \\
&\lesssim&  {\left(\int_r^{+\infty} \rho {|U|}^2 d \rho \right)}^{\frac{1}{2}}{\left(\int_r^{+\infty} \rho {|U_\rho|}^2 d \rho \right)}^{\frac{1}{2}}\\
&\leq& C.
\end{eqnarray*}
The last of the above inequality holds according to the energy estimates and \eqref{3.20}.

Then, using \eqref{3.29}, we obtain
\begin{eqnarray*}
{{\mathcal{L}}_0}^2 &\lesssim& r \mathbf{f} \mathbf{e_0}+r^2 \mathbf{e_0} \mathbf{f}+\mathbf{f}^2+\frac{1}{4} \mathbf{e_0}^2-\mathbf{m}^2 \\
&\lesssim&\frac{1}{4} \mathbf{e_0}^2+\frac{1}{2} \mathbf{e_0} \mathbf{f}+\frac{1}{4}\mathbf{f}^2-\mathbf{m}^2 \\
&\lesssim&\mathbf{e}^2-\mathbf{m}^2 \\
&\lesssim& \frac{{\hat{\mathcal{F}}}^2{\hat{\mathcal{G}}}^2}{r^2}.
\end{eqnarray*}
The rest of the proof is same to what Andersson did in \cite{Andersson}, so we omit it here.
\end{proof}

\subsection{Local spacetime integral estimates}
Using Proposition \ref{prop3.5} and Proposition \ref{prop3.6}, we first give the non-concentration of integrated kinetic energy, of which the proof is similar to \cite{Andersson}.
\begin{proposition} \label{prop3.7}
For the kinetic energy density defined as follows
\begin{displaymath}
\mathbf{e}_{kin}:=\frac{1}{2}e^{-2\alpha}(2{\gamma_t}^2+{\phi_t}^2),
\end{displaymath}
the spacetime integral of $\mathbf{e}_{kin}$ does not concentrate in the past null cone of $O$, i.e,
\begin{displaymath}
\frac{1}{r_2(\tau)}\int_{K_\tau}\mathbf{e}_{kin} \bar{\mu}_g \to 0,\quad as \quad \tau \to 0
\end{displaymath}
where $r_2(\tau)$ is the radius function appearing in Proposition \ref{prop3.6}.
\end{proposition}
Finally, we estimate the radial potential energy as below.
\begin{proposition} \label{prop3.8}
The spacetime integral of the radial potential energy does not concentrate in the past null cone of $O$, i.e,
\begin{displaymath}
\frac{1}{r_2(\tau)}\int_{K_\tau}\frac{1}{2}e^{-2\beta}(2{\gamma_r}^2+{\phi_r}^2) \bar{\mu}_g \to 0,\quad as \quad \tau \to 0.
\end{displaymath}
\end{proposition}
\begin{proof}
Let us construct the vector $P^{\nu}_{t_1}$ such that
\begin{displaymath}
P^{\nu}_{t_1}:=2 \gamma^{\nu} (\gamma-\gamma(-t_1,r_2(-t_1)))+\phi^{\nu} (\phi-\phi(-t_1,r_2(-t_1))).
\end{displaymath}
where $t_1>0$ is a small parameter. Then the divergence is given by
\begin{eqnarray*}
\nabla_{\nu} P^{\nu}_{t_1}&=&2(\Box_g \gamma)(\gamma-\gamma(-t_1,r_2(-t_1)))+(\Box_g \phi) (\phi-\phi(-t_1,r_2(-t_1)))+2\gamma^{\nu} \gamma_{\nu}+\phi^{\nu}\phi_{\nu} \\
&=& -m^2e^{-2\gamma} \phi^2 (\gamma-\gamma(-t_1,r_2(-t_1)))+m^2 e^{-2\gamma} \phi(\phi-\phi(-t_1,r_2(-t_1)))\\
&&+2\gamma^{\nu} \gamma_{\nu}+\phi^{\nu}\phi_{\nu}.
\end{eqnarray*}
Applying Stokes' theorem on $K(-t_1,-t_1^3)$,
\begin{equation} \label{3.30}
\int_{K(-t_1,-t_1^3)} \nabla_{\nu} P^{\nu}_{t_1} \bar{\mu}_g=\int_{\Sigma^O_{-t_1^3}} e^{\alpha} P_{t_1}^t \bar{\mu}_q-\int_{\Sigma^O_{-t_1}} e^{\alpha} P_{t_1}^t \bar{\mu}_q+Flux(P_{t_1})(-t_1,-t_1^3).
\end{equation}
By the energy estimates in \cite{chen}, \eqref{3.21} and \eqref{3.29}, we have
\begin{eqnarray} \label{3.31}
\int_{\Sigma^O_{-t_1^3}} e^{\alpha} P_{t_1}^t \bar{\mu}_q&\lesssim & \int_0^{r_2(-t_1^3)} |U_t|\cdot|U-U(-t_1,r_2(-t_1))| rdr \\ \nonumber
& \lesssim & \int_0^{r_2(-t_1^3)} |U_t|\cdot|U(-t_1,r_2(-t_1))| rdr+ \int_0^{r_2(-t_1^3)} |U_t|\cdot|U| rdr\\ \nonumber
& \lesssim & E_0 {t_1}^{-\frac{1}{2}}{\left(\int_0^{r_2(-t_1^3)} rdr\right)}^{\frac{1}{2}}+E_0 {\left(\int_0^{r_2(-t_1^3)} |U|^2 rdr\right)}^{\frac{1}{2}}\\ \nonumber
& \lesssim & {t_1}^{-\frac{1}{2}}{\left(\int_0^{r_2(-t_1^3)} rdr\right)}^{\frac{1}{2}}+ {\left(\int_0^{r_2(-t_1^3)} |U|^2 rdr\right)}^{\frac{1}{2}}\\ \nonumber
& \lesssim & {t_1}^{\frac{5}{2}}+ {t_1}^{\frac{3}{2}}\\ \nonumber
& \lesssim & {r_2(-t_1)}^{\frac{3}{2}}.
\end{eqnarray}
For the second term in \eqref{3.30}, we have the following estimates.
\begin{eqnarray*}
-\int_{\Sigma^O_{-t_1}} e^{\alpha} P_{t_1}^t \bar{\mu}_q&\lesssim & \int_0^{r_2(-t_1)} |U_t|\cdot|U-U(-t_1,r_2(-t_1))| rdr \\
& \lesssim &   {\left(\int_0^{r_2(-t_1)} |U_t|^2 r^{1+a}dr\right)}^{\frac{1}{2}}{\left(\int_0^{r_2(-t_1)} (U-U(-t_1,r_2(-t_1)))^2 r^{1-a}dr\right)}^{\frac{1}{2}}\\
& \lesssim &  {\left(\int_0^{r_2(-t_1)} \mathbf{e} r^{1+a}dr\right)}^{\frac{1}{2}}{\left(\int_0^{r_2(-t_1)} (U-U(-t_1,r_2(-t_1)))^2 r^{1-a}dr\right)}^{\frac{1}{2}}
\end{eqnarray*}
where $0<a<1$ is an arbitrary constant.

Similar to \cite{Andersson}, using Proposition \ref{prop3.6}, we have
\begin{displaymath}
{\left(\int_0^{r_2(-t_1)} \mathbf{e} r^{1+a}dr\right)}^{\frac{1}{2}} \lesssim  \varepsilon{r_2(-t_1)}^{\frac{a}{2}}
\end{displaymath}
for any small enough $\varepsilon>0$.

Then integrating by parts, using \eqref{3.29}, we obtain
\begin{eqnarray*}
\int_0^{r_2(-t_1)} (U-U(-t_1,r_2(-t_1)))^2 r^{1-a}dr&\lesssim & \int_0^{r_2(-t_1)} (U-U(-t_1,r_2(-t_1)))^2 dr^{2-a} \\
& = & -2\int_0^{r_2(-t_1)} (U-U(-t_1,r_2(-t_1))) U_r r^{2-a}dr \\
& \lesssim & {\left(\int_0^{r_2(-t_1)} {|U_r|}^2 r^{3-a}dr\right)}^{\frac{1}{2}} {\left(\int_0^{r_2(-t_1)} (U-U(-t_1,r_2(-t_1)))^2 r^{1-a}dr\right)}^{\frac{1}{2}}\\
& \lesssim & E_0 {r_2(-t_1)}^{1-\frac{a}{2}} {\left(\int_0^{r_2(-t_1)} (U-U(-t_1,r_2(-t_1)))^2 r^{1-a}dr\right)}^{\frac{1}{2}}
\end{eqnarray*}
which implies
\begin{displaymath}
{\left(\int_0^{r_2(-t_1)} (U-U(-t_1,r_2(-t_1)))^2 r^{1-a}dr\right)}^{\frac{1}{2}} \lesssim  E_0{r_2(-t_1)}^{1-\frac{a}{2}}.
\end{displaymath}
Thus, we have for the second integral in \eqref{3.30},
\begin{equation} \label{3.32}
-\int_{\Sigma^O_{-t_1}} e^{\alpha} P_{t_1}^t \bar{\mu}_q \lesssim  \varepsilon{r_2(-t_1)}.
\end{equation}
Let us consider the flux of $P_{t_1}$ through the null surface $C(-t_1,-t_1^3)$ now. By Proposition \ref{prop3.4}, Poincar\'{e}'s inequality and \cite{chen}, we have
\begin{eqnarray} \label{3.33}
 \nonumber Flux(P_{t_1})(-t_1,-t_1^3)&= & \int_{C(-t_1,-t_1^3)} dv(P_{t_1}) \bar{\mu}_v \\
& \lesssim & \int_{C(-t_1,-t_1^3)} |U(u,0)-U(-t_1-r_2(-t_1),0)|\cdot |U_u| rdu \\ \nonumber
& \lesssim & \varepsilon {\left(\int_{C(-t_1,-t_1^3)} (U(u,0)-U(-t_1-r_2(-t_1),0))^2 Rdu\right)}^{\frac{1}{2}}\\ \nonumber
& \lesssim & \varepsilon r_2(-t_1){\left(\int_{C(-t_1,-t_1^3)} {U_u}^2(u,0) Rdu\right)}^{\frac{1}{2}}\\ \nonumber
& \lesssim & \varepsilon {r_2(-t_1)}.
\end{eqnarray}
Now, if we go back to Stokes' theorem \eqref{3.30} and use the estimates \eqref{3.29} \eqref{3.31} \eqref{3.32} and \eqref{3.33}, we get
\begin{eqnarray*}
\int_{K(-t_1,-t_1^3)}\frac{1}{2}e^{-2\beta}(2{\gamma_r}^2+{\phi_r}^2) \bar{\mu}_g&\lesssim&\varepsilon {r_2(-t_1)}+{r_2(-t_1)}^{\frac{3}{2}}+\int_{K(-t_1,-t_1^3)}\mathbf{e}_{kin} \bar{\mu}_g \\
&&+\int_{K(-t_1,-t_1^3)}\phi^2 |\gamma-\gamma(-t_1,r_2(-t_1))|\bar{\mu}_g\\
&&+\int_{K(-t_1,-t_1^3)}|\phi||\phi-\phi(-t_1,r_2(-t_1))|\bar{\mu}_g \\
&\lesssim&\varepsilon {r_2(-t_1)}+{r_2(-t_1)}^{\frac{3}{2}}+\int_{K(-t_1,-t_1^3)}\mathbf{e}_{kin} \bar{\mu}_g \\
&&+\int_{K(-t_1,-t_1^3)}U^4\bar{\mu}_g+\int_{K(-t_1,-t_1^3)}U^2\bar{\mu}_g\\
&&+\int_{K(-t_1,-t_1^3)}U(-t_1,r_2(-t_1))^2\bar{\mu}_g \\
&\lesssim&\varepsilon {r_2(-t_1)}+{r_2(-t_1)}^{\frac{3}{2}}+\int_{K(-t_1,-t_1^3)}\mathbf{e}_{kin} \bar{\mu}_g \\
&&+\int_{K(-t_1,-t_1^3)}U^4\bar{\mu}_g+\int_{K(-t_1,-t_1^3)}U^2\bar{\mu}_g\\
&&+\int_{K(-t_1,-t_1^3)}r^{-\frac{1}{2}}\bar{\mu}_g \\
&\lesssim&\varepsilon {r_2(-t_1)}+{r_2(-t_1)}^{\frac{3}{2}}+\int_{K(-t_1,-t_1^3)}\mathbf{e}_{kin} \bar{\mu}_g \\
&&+\int_{K(-t_1,-t_1^3)}U^4\bar{\mu}_g+\int_{K(-t_1,-t_1^3)}U^2\bar{\mu}_g.
\end{eqnarray*}

Let $t_1 \to 0$, the above term tends to $0$ according to Proposition \ref{prop3.5} and Proposition \ref{prop3.7}. This concludes the proof of the Proposition \ref{prop3.8}.
\end{proof}

Therefore, combining Proposition \ref{prop3.5}, Proposition \ref{prop3.7} and Proposition \ref{prop3.8}, we have finished the proof of Theorem \ref{thm3.1}.

\section{Small energy implies global regularity}
We consider the Cauchy problem for the 2+1 dimensional radially symmetric Einstein-wave-Klein-Gordon system  \eqref{1.6}-\eqref{1.11}  on $(M,g)$. For $U=(\gamma,\phi)$, our goal in this section is to prove the following theorem:
\begin{theorem} \label{thm5.1}
There exists a small enough $\varepsilon>0$, such that for any regular Cauchy data of energy $E_0<\varepsilon$, the Cauchy problem for $(\ref{1.6})$-$(\ref{1.11})$ admits an unique globally smooth solution.
\end{theorem}

\subsection{Estimates for $\partial_v U$}

By \cite{chen}, it is clear that away from the axis we have a (1,1)-dimensional problem, for which the global regularity is known, therefore the first singularity occurs at the centre. Thus, we assume that the first possible singularity takes place at $(T^*,0)$.

Then for $\bar{T}$ close enough to $T^*$, by shifting time, we can consider the problem on
\begin{displaymath}
K_0=\left\{ (T,R)|0 \leq R < T^*-\bar{T}-T,\  T \geq 0  \right\}.
\end{displaymath}
Note that
\begin{displaymath}
T=\frac{u+v}{2},\ R=\frac{v-u}{2}
\end{displaymath}
\begin{displaymath}
\partial_u=\frac{\partial_T-\partial_R}{2},\ \partial_v=\frac{\partial_T+\partial_R}{2}
\end{displaymath}

Let
\begin{displaymath}
X \triangleq \sup\limits_{(T,R) \in K_0} r^{\delta} | U_{v}|
\end{displaymath}
where $ \frac{1}{2}<\delta<\frac{2}{3}$.

 We will show that $X$ is bounded. Before doing this, we give a lemma firstly.
\begin{lemma} \label{prop5.2}
Let $\mu>0,\ a,b>0,\ b<1,\ a+b>1$, then
\begin{displaymath}
\int_{0}^{+\infty} {(\mu+x)}^{-a} x^{-b} dx \leq C \mu^{-(a+b-1)}.
\end{displaymath}
\end{lemma}
\begin{proof}
\begin{eqnarray*}
\int_{0}^{+\infty} {(\mu+x)}^{-a} x^{-b} dx & \leq & \int_{0}^{\mu} {(\mu+x)}^{-a} x^{-b} dx+\int_{\mu}^{+\infty} {(\mu+x)}^{-a} x^{-b} dx \\
& \leq & \mu^{-a} \int_{0}^{\mu}  x^{-b} dx+\int_{\mu}^{+\infty} x^{-(a+b)} dx \\
&=& {(1-b)}^{-1} \mu^{-(a+b-1)}+{(a+b-1)}^{-1} \mu^{-(a+b-1)} \\
&=&C \mu^{-(a+b-1)}.
\end{eqnarray*}
\end{proof}

Now that we consider the wave equations \eqref{1.18} on the flat spacetime, we can define the Cartesian coordinate correspondingly and we denote it by $x$. In addition, by \cite{chen}, without loss of generality we can regard $R$ as $r$ in the following part. Then, we have the following lemma, which plays a central role in our subsequent proof.
\begin{lemma} \label{lem5.3}
There exists a constant $A>0$ depending on the initial data in $K_0$ such that the following estimate holds for any $(T_0,x_0)$ in $K_0$,
\begin{equation} \label{5.14}
|U(T_0+|x_0|,0)-U(T_0,x_0)| \lesssim {r}^{1-\delta}(T_0,|x_0|) \left(A+\varepsilon X\right).
\end{equation}
where $x_0$ denotes the corresponding Cartesian coordinates.
\end{lemma}
\begin{proof}
First, for $(\ref{1.18})$, we have that
\begin{displaymath}
U(T_0+|x_0|,0)=U^0 (T_0+|x_0|,0)+\frac{1}{2\pi}\int \int_{|y| \leq T_0+|x_0|-\tau} \frac{h(\tau,y)}{\sqrt{{(T_0+|x_0|-\tau)}^2-{|y|}^2}} dy d \tau
\end{displaymath}
\begin{displaymath}
U(T_0,x_0)=U^0 (T_0,x_0)+\frac{1}{2\pi}\int \int_{|x_0-y| \leq T_0-\tau} \frac{h(\tau,y)}{\sqrt{{(T_0-\tau)}^2-{|x_0-y|}^2}} dy d \tau
\end{displaymath}
where $U^0$ is the solution to the linear problem $\Box_m {U^{0}}=0$. Then we obtain
\begin{eqnarray*}
U(T_0+|x_0|,0)-U(T_0,x_0) &= & U^0(T_0+|x_0|,0)-U^0(T_0,x_0)\\
&&+\frac{1}{2\pi} \iint\nolimits_{|x_0-y|>T_0-\tau,|y| \leq T_0+|x_0|-\tau} \frac{h(\tau,y)}{\sqrt{{(T_0+|x_0|-\tau)}^2-{|y|}^2}} dy d \tau \\
&&+\frac{1}{2\pi} \iint\nolimits_{|x_0-y| \leq T_0-\tau} \left(\frac{1}{\sqrt{{(T_0+|x_0|-\tau)}^2-{|y|}^2}} \right. \\
&&- \left. \frac{1}{\sqrt{{(T_0-\tau)}^2 -{|x_0-y|}^2}} \right) h(\tau,y) dy d \tau  \\
& \triangleq & \uppercase\expandafter{\romannumeral1}+\uppercase\expandafter{\romannumeral2}+
\uppercase\expandafter{\romannumeral3}.
\end{eqnarray*}
By the smoothness of $U^0$,
\begin{equation*}
|\uppercase\expandafter{\romannumeral1}|= |U^0(T_0+|x_0|,0)-U^0(T_0,x_0)| \leq A  |x_0| \lesssim A r(T_0,|x_0|) \lesssim A r^{1-\delta}(T_0,|x_0|).
\end{equation*}
In null coordinates, we have
\begin{eqnarray*}
|\uppercase\expandafter{\romannumeral2}| & \leq & \int_{T_0-|x_0|}^{T_0+|x_0|} \int_{-v}^{v} \frac{|h(u,v)| R}{\sqrt{(T_0+|x_0|-u)(T_0+|x_0|-v)}} du dv \\
& \lesssim & \int_{T_0-|x_0|}^{T_0+|x_0|} \int_{-v}^{v} \frac{|\mathcal{H}_1| R}{\sqrt{(T_0+|x_0|-u)(T_0+|x_0|-v)}} du dv \\
&&+\int_{T_0-|x_0|}^{T_0+|x_0|} \int_{-v}^{v} \frac{|\mathcal{H}_2| R}{\sqrt{(T_0+|x_0|-u)(T_0+|x_0|-v)}} du dv \\
&&+\int_{T_0-|x_0|}^{T_0+|x_0|} \int_{-v}^{v} \frac{|\frac{2R r_u+r}{r} U_v| }{\sqrt{(T_0+|x_0|-u)(T_0+|x_0|-v)}} du dv \\
&&+\int_{T_0-|x_0|}^{T_0+|x_0|} \int_{-v}^{v} \frac{|\frac{2R r_v-r}{r} U_u| }{\sqrt{(T_0+|x_0|-u)(T_0+|x_0|-v)}} du dv \\
& \triangleq &  \uppercase\expandafter{\romannumeral2}_1+\uppercase\expandafter{\romannumeral2}_2+\uppercase\expandafter{\romannumeral2}_3+\uppercase\expandafter{\romannumeral2}_4
\end{eqnarray*}

For $\uppercase\expandafter{\romannumeral2}_1$, we have
\begin{eqnarray*}
\uppercase\expandafter{\romannumeral2}_1 & \lesssim & \int_{T_0-|x_0|}^{T_0+|x_0|} \int_{-v}^{v} \frac{ {\phi}^2 R}{\sqrt{(T_0+|x_0|-u)(T_0+|x_0|-v)}} du dv \\
& \lesssim & \int_{T_0-|x_0|}^{T_0+|x_0|} \frac{1}{\sqrt{T_0+|x_0|-v}} dv {\left( \int_{-v}^{v}  {\phi}^{4} R du \right)}^{\frac{1}{2}} {\left( \int_{-v}^{v} \frac{R}{T_0+|x_0|-u} du \right)}^{\frac{1}{2}} \\
& \lesssim & \int_{T_0-|x_0|}^{T_0+|x_0|} \frac{1}{\sqrt{T_0+|x_0|-v}} dv {\left( \int_{-2\bar{T}-v}^{v}  {\phi}^{4} R du \right)}^{\frac{1}{2}} {\left( \int_{-v}^{v} \frac{1}{R^{2\delta-1}(T_0+|x_0|-u)} du \right)}^{\frac{1}{2}} \\
& \lesssim & \int_{T_0-|x_0|}^{T_0+|x_0|} \frac{1}{\sqrt{T_0+|x_0|-v}} dv {\left( \int_{0}^{v+\bar{T}}  {\phi}^{4}(v-2R,v) R dR \right)}^{\frac{1}{2}} {\left( \int_{-v}^{v} \frac{1}{R^{2\delta-1}(T_0+|x_0|-u)} du \right)}^{\frac{1}{2}}
\end{eqnarray*}
using the U(1) symmetry, energy estimates on the flux and Sobolev's inequality, we get
\begin{eqnarray*}
{\left( \int_{0}^{v+\bar{T}} {\phi}^{4}(v-2R,v) R dR \right)}^{\frac{1}{4}} &=& {\left( \frac{1}{2\pi} \int_{0}^{2\pi} \int_{0}^{v+\bar{T}} {\phi}^{4}(v-2R,v) R dR d\theta \right)}^{\frac{1}{4}} \\ \nonumber
& \lesssim & {\left( \int_{-2\bar{T}-v}^{v} {\phi_u}^{2} R du \right)}^{\frac{1}{2}}+{\left( \int_{-2\bar{T}-v}^{v} {\phi}^{2} R du \right)}^{\frac{1}{2}} \\ \nonumber
& \lesssim & C+{\left( \int_{-2\bar{T}-v}^{v} {\phi}^{2} R du \right)}^{\frac{1}{2}}.
\end{eqnarray*}
Now we estimate the second integral, by the regularity of initial data, the energy estimates on flux, and Cauchy-Schwartz inequality, we have
\begin{eqnarray*}
 \int_{-2\bar{T}-v}^{v} {\phi}^{2} R du
 & \lesssim&  \int_{0}^{v+\bar{T}} {\phi}^{2}(v-2R,v) R dR\\
 &\lesssim& \int_{0}^{v+\bar{T}} {\phi}^{2}(v-2R,v) dR^2  \\
&=& {(v+\bar{T})}^2 \phi^2(-\bar{T},v+\bar{T})+\int_{-2\bar{T}-v}^v R^2 \phi \phi_u du \\
& \leq & C+C(T^*)\int_{2T_0-v}^v R |\phi| |\phi_u| du \\
& \lesssim & C+\frac{1}{8}\int_{-2\bar{T}-v}^v R \phi^2 du+2\int_{-2\bar{T}-v}^v R {\phi_u}^2 du \\
& \leq & C+\frac{1}{8}\int_{-2\bar{T}-v}^v R \phi^2 du
\end{eqnarray*}
where $-\bar{T}$ is the initial time in Theorem \ref{thm1.6}. This implies
\begin{equation*}
{\|\phi\|}_{L^2(Rdu)}  \leq   C.
\end{equation*}
Thus,
\begin{equation} \label{5.15}
{\|\phi\|}_{L^4(Rdu)}  \leq  C.
\end{equation}
Then, by Lemma \ref{prop5.2},
\begin{eqnarray} \label{5.16}
\int_{-v}^{v} \frac{1}{R^{2\delta-1}(T_0+|x_0|-u)} du &=& \int_{0}^{v} \frac{1}{R^{2\delta-1}(T_0+|x_0|-v+2R)} dR \\ \nonumber
& \lesssim & \frac{1}{{(T_0+|x_0|-v)}^{2\delta-1}}
\end{eqnarray}
thus,
\begin{eqnarray*}
\uppercase\expandafter{\romannumeral2}_1 & \lesssim &   \int_{T_0-|x_0|}^{T_0+|x_0|} \frac{1}{{(T_0+|x_0|-v)}^{\delta}} dv \\ \nonumber
& \lesssim &  {|x_0|}^{1-\delta}\lesssim r^{1-\delta}(u_0,v_0).
\end{eqnarray*}
Similarly, we have, for $\uppercase\expandafter{\romannumeral2}_2$,
\begin{eqnarray*}
\uppercase\expandafter{\romannumeral2}_2 & \lesssim & \int_{T_0-|x_0|}^{T_0+|x_0|} \int_{-v}^{v} \frac{ e^{-2\gamma}{\phi} R}{\sqrt{(T_0+|x_0|-u)(T_0+|x_0|-v)}} du dv \\
& \lesssim & \int_{T_0-|x_0|}^{T_0+|x_0|} \frac{1}{\sqrt{T_0+|x_0|-v}} dv {\left( \int_{-v}^{v} e^{-2\gamma} {\phi}^{2} R du \right)}^{\frac{1}{2}} {\left( \int_{-v}^{v} \frac{R}{T_0+|x_0|-u} du \right)}^{\frac{1}{2}} \\
& \lesssim &\varepsilon \int_{T_0-|x_0|}^{T_0+|x_0|} \frac{1}{\sqrt{T_0+|x_0|-v}} dv  {\left( \int_{-v}^{v} \frac{1}{R^{2\delta-1}(T_0+|x_0|-u)} du \right)}^{\frac{1}{2}} \\
& \lesssim &  {|x_0|}^{1-\delta}\lesssim r^{1-\delta}(u_0,v_0).
\end{eqnarray*}
Next, we will estimate $\uppercase\expandafter{\romannumeral2}_3$,
\begin{equation*}
\uppercase\expandafter{\romannumeral2}_3  \lesssim  X\int_{T_0-|x_0|}^{T_0+|x_0|} \frac{1}{\sqrt{T_0+|x_0|-v}}dv \int_{-v}^{v} \frac{|\frac{2R r_u+r}{r} |}{r^{\delta} \sqrt{T_0+|x_0|-u}} du.
\end{equation*}
Noting that by \eqref{5.15},
\begin{displaymath}
|\frac{2R r_u+r}{r} | \lesssim  R,
\end{displaymath}
and
\begin{displaymath}
|\frac{2R r_v-r}{r} | \lesssim  R,
\end{displaymath}
by Lemma \ref{prop5.2},
\begin{eqnarray*}
\uppercase\expandafter{\romannumeral2}_3  & \lesssim & \varepsilon X\int_{T_0-|x_0|}^{T_0+|x_0|} \frac{1}{{(T_0+|x_0|-v)}^{\delta}}dv \\ \nonumber
& \lesssim & \varepsilon X {|x_0|}^{1-\delta}\lesssim \varepsilon X r^{1-\delta}(u_0,v_0).
\end{eqnarray*}
Similarly, we can estimate $\uppercase\expandafter{\romannumeral2}_4$ as above,
\begin{eqnarray*}
\uppercase\expandafter{\romannumeral2}_4 &\lesssim & \int_{T_0-|x_0|}^{T_0+|x_0|} \frac{1}{\sqrt{T_0+|x_0|-v}}dv \int_{-v}^{v} \frac{r|U_u |}{ \sqrt{T_0+|x_0|-u}} du \\
& \lesssim& \varepsilon \int_{T_0-|x_0|}^{T_0+|x_0|} \frac{1}{\sqrt{T_0+|x_0|-v}}dv{\left( \int_{-v}^{v} \frac{r}{(T_0+|x_0|-u)} du \right)}^{\frac{1}{2}} \\
&\lesssim& \int_{T_0-|x_0|}^{T_0+|x_0|} \frac{1}{\sqrt{T_0+|x_0|-v}}dv{\left( \int_{-v}^{v} \frac{1}{R^{2\delta-1}(T_0+|x_0|-u)} du \right)}^{\frac{1}{2}} \\
& \lesssim &r^{1-\delta}(u_0,v_0).
\end{eqnarray*}

To estimate $ \uppercase\expandafter{\romannumeral3}$, we break it into two parts,
\begin{eqnarray*}
\uppercase\expandafter{\romannumeral3} &=& -\frac{1}{2\pi} \iint\nolimits_{|x_0-y| \leq T_0-\tau} \left(\frac{1}{\sqrt{{(T_0-\tau)}^2 -{|x_0-y|}^2}}
-  \frac{1}{\sqrt{{(T_0+|x_0|-\tau)}^2-{|y|}^2}} \right) h(\tau,y) dy d \tau \\
&=& -\frac{1}{2\pi}\iint\nolimits_{|x_0|+|y| \leq T_0-\tau} \left(\frac{1}{\sqrt{{(T_0-\tau)}^2 -{|x_0-y|}^2}}
-  \frac{1}{\sqrt{{(T_0+|x_0|-\tau)}^2-{|y|}^2}} \right) h(\tau,y) dy d \tau \\
&-&\frac{1}{2\pi}\iint\nolimits_{|x_0-y| \leq T_0-\tau , |x_0|+|y| > T_0-\tau} \left(\frac{1}{\sqrt{{(T_0-\tau)}^2 -{|x_0-y|}^2}}
-  \frac{1}{\sqrt{{(T_0+|x_0|-\tau)}^2-{|y|}^2}} \right) h(\tau,y) dy d \tau \\
& \triangleq & {\uppercase\expandafter{\romannumeral3}}_A+{\uppercase\expandafter{\romannumeral3}}_B.
\end{eqnarray*}
We can similarly define ${\uppercase\expandafter{\romannumeral3}}_{Ai}$ and ${\uppercase\expandafter{\romannumeral3}}_{Bi},\ i=1,\cdots,4$ as we did for $\uppercase\expandafter{\romannumeral2}$.

For ${\uppercase\expandafter{\romannumeral3}}_A$, we note that
\begin{flalign*}
&\frac{1}{\sqrt{{(T_0-\tau)}^2 -{|x_0-y|}^2}}
-  \frac{1}{\sqrt{{(T_0+|x_0|-\tau)}^2-{|y|}^2}}& \\
&=\frac{\left({(T_0+|x_0|-\tau)}^2-{|y|}^2\right)-\left({(T_0-\tau)}^2-{|x_0-y|}^2\right)}{\sqrt{{(T_0-\tau)}^2-{|x_0-y|}^2} \sqrt{{(T_0+|x_0|-\tau)}^2-{|y|}^2}\left(\sqrt{{(T_0-\tau)}^2-{|x_0-y|}^2}+\sqrt{{(T_0+|x_0|-\tau)}^2-{|y|}^2}\right)}& \\
&\leq \frac{(T_0+|x_0|-\tau+|y|)(T_0+|x_0|-\tau-|y|)-(T_0-\tau+|x_0-y|)(T_0-\tau-|x_0-y|)}{\left( {(T_0+|x_0|-\tau)}^2-{|y|}^2 \right) \sqrt{{(T_0-\tau)}^2-{|x_0-y|}^2}} & \\
&=\frac{\left[ T_0+|x_0|-\tau+|y|-(T_0-\tau+|x_0-y|) \right](T_0+|x_0|-\tau-y)}{{\left( {(T_0+|x_0|-\tau)}^2-{|y|}^2 \right) \sqrt{{(T_0-\tau)}^2-{|x_0-y|}^2}} }& \\
&+\frac{(T_0-\tau+|x_0-y|)\left[T_0+|x_0|-\tau-|y|-(T_0-\tau-|x_0-y|)\right]}{{\left( {(T_0+|x_0|-\tau)}^2-{|y|}^2 \right) \sqrt{{(T_0-\tau)}^2-{|x_0-y|}^2}} }&\\
&=\frac{(|x_0|+|y|-|x_0-y|)(T_0+|x_0|-\tau-|y|)+(T_0-\tau+|x_0-y|)(|x_0-y|+|x_0|-|y|)}{{\left( {(T_0+|x_0|-\tau)}^2-{|y|}^2 \right) \sqrt{{(T_0-\tau)}^2-{|x_0-y|}^2}}} & \\
& \lesssim \frac{|x_0|(T_0-\tau+|x_0-y|)}{{\left( {(T_0+|x_0|-\tau)}^2-{|y|}^2 \right) \sqrt{{(T_0-\tau)}^2-{|x_0-y|}^2}}} & \\
& \lesssim \frac{|x_0|\sqrt{T_0-\tau+|x_0-y|}}{\left( {(T_0+|x_0|-\tau)}^2-{|y|}^2 \right) \sqrt{T_0-\tau-|x_0-y|}}& \\
& \lesssim \frac{|x_0|}{\sqrt{T_0+|x_0|-\tau+|y|} (T_0+|x_0|-\tau-|y|)\sqrt{T_0-\tau-|x_0|-|y|}} &
\end{flalign*}
Thus, using $(\ref{5.15})$ and Lemma \ref{prop5.2},
\begin{eqnarray*}
{\uppercase\expandafter{\romannumeral3}}_{A1} & \lesssim & |x_0| \iint\frac{{\phi}^2 r}{\sqrt{T_0+|x_0|-u}(T_0+|x_0|-v) \sqrt{T_0-|x_0|-v}} du dv \\
& \lesssim &  |x_0| \int_{0}^{T_0-|x_0|} \frac{1}{(T_0+|x_0|-v)\sqrt{T_0-|x_0|-v}} dv {\left( \int_0^{v} \frac{r}{(T_0+|x_0|-u) } dR \right)}^{\frac{1}{2}} \\
& \lesssim &  |x_0| \int_{0}^{T_0-|x_0|} \frac{1}{(T_0+|x_0|-v)\sqrt{T_0-|x_0|-v}} dv {\left( \int_0^{v} \frac{1}{R^{2\delta-1}(T_0+|x_0|-u) } dR \right)}^{\frac{1}{2}} \\
& \lesssim & |x_0| \int_{0}^{T_0-|x_0|} \frac{1}{{(T_0+|x_0|-v)}^{\delta+\frac{1}{2}}\sqrt{T_0-|x_0|-v}} dv \\
& \lesssim & |x_0|^{1-\delta} \lesssim r^{1-\delta}(u_0,v_0)
\end{eqnarray*}
Similarly, we can obtain the same estimates for ${\uppercase\expandafter{\romannumeral3}}_{A2}$,
\begin{eqnarray*}
{\uppercase\expandafter{\romannumeral3}}_{A2} & \lesssim & |x_0| \iint\frac{e^{-2\gamma}{\phi} r}{\sqrt{T_0+|x_0|-u}(T_0+|x_0|-v) \sqrt{T_0-|x_0|-v}} du dv \\
& \lesssim &  |x_0| \int_{0}^{T_0-|x_0|} \frac{1}{(T_0+|x_0|-v)\sqrt{T_0-|x_0|-v}} dv {\left( \int_0^{v} \frac{r}{(T_0+|x_0|-u) } dR \right)}^{\frac{1}{2}} \\
& \lesssim &  |x_0| \int_{0}^{T_0-|x_0|} \frac{1}{(T_0+|x_0|-v)\sqrt{T_0-|x_0|-v}} dv {\left( \int_0^{v} \frac{1}{R^{2\delta-1}(T_0+|x_0|-u) } dR \right)}^{\frac{1}{2}} \\
& \lesssim & |x_0| \int_{0}^{T_0-|x_0|} \frac{1}{{(T_0+|x_0|-v)}^{\delta+\frac{1}{2}}\sqrt{T_0-|x_0|-v}} dv \\
& \lesssim & |x_0|^{1-\delta} \lesssim r^{1-\delta}(u_0,v_0)
\end{eqnarray*}
For ${\uppercase\expandafter{\romannumeral3}}_{A3}$, by Lemma \ref{prop5.2},
\begin{eqnarray*}
{\uppercase\expandafter{\romannumeral3}}_{A3} & \lesssim & \varepsilon X |x_0| \int_{0}^{T_0-|x_0|} \frac{1}{(T_0+|x_0|-v)\sqrt{T_0-|x_0|-v}} dv  \int_{-v}^{v} \frac{1}{r^{\delta} \sqrt{T_0+|x_0|-u} } du \\
& \lesssim & \varepsilon X |x_0| \int_{0}^{T_0-|x_0|} \frac{1}{{(T_0+|x_0|-v)}^{\delta+\frac{1}{2}} \sqrt{T_0-|x_0|-v}} dv \\
& \lesssim & \varepsilon X r^{1-\delta}(u_0,v_0)
\end{eqnarray*}
then, we can get
\begin{eqnarray*}
{\uppercase\expandafter{\romannumeral3}}_{A4} & \lesssim& |x_0| \int_{0}^{T_0-|x_0|} \frac{1}{(T_0+|x_0|-v)\sqrt{T_0-|x_0|-v}} dv \int_{-v}^{v} \frac{r |U_u|}{ \sqrt{T_0+|x_0|-u} } du \\
&\lesssim&   |x_0| \int_{0}^{T_0-|x_0|} \frac{1}{{(T_0+|x_0|-v)}^{\delta+\frac{1}{2}} \sqrt{T_0-|x_0|-v}} dv \\
&\lesssim & r^{1-\delta}(u_0,v_0).
\end{eqnarray*}

 Finally, we estimate ${\uppercase\expandafter{\romannumeral3}}_B$. We have
\begin{eqnarray*}
{\uppercase\expandafter{\romannumeral3}}_B &=&-\frac{1}{2\pi}\iint\nolimits_{|x_0-y| \leq T_0-\tau , |x_0|+|y| > T_0-\tau} \frac{1}{\sqrt{{(T_0-\tau)}^2 -{|x_0-y|}^2}}
 h(\tau,y) dy d \tau \\
 &&+\frac{1}{2\pi}\iint\nolimits_{|x_0-y| \leq T_0-\tau , |x_0|+|y| > T_0-\tau}  \frac{1}{\sqrt{{(T_0+|x_0|-\tau)}^2-{|y|}^2}} h(\tau,y) dy d \tau \\
 & \triangleq & {\uppercase\expandafter{\romannumeral3}}_B^{(1)}+{\uppercase\expandafter{\romannumeral3}}_B^{(2)}.
\end{eqnarray*}
Estimates for ${\uppercase\expandafter{\romannumeral3}}_B^{(2)}$ can be obtained in the same way as we did for $\uppercase\expandafter{\romannumeral2}$.
What left to be estimated is ${\uppercase\expandafter{\romannumeral3}}_B^{(1)}$.

For ${\uppercase\expandafter{\romannumeral3}}_B^{(1)}$, we have
\begin{eqnarray*}
|{\uppercase\expandafter{\romannumeral3}}_B^{(1)}| &\lesssim& \int_{T_0-|x_0|}^{T_0+|x_0|} \int_{-v}^{T_0-|x_0|} |h(u,v)|r du dv  \\
&&\int_{\sin^2{\frac{\varphi}{2}} \leq \frac{{(T_0-\tau)}^2-{(|x_0|-|y|)}^2}{4|x_0||y|}} \frac{1}{\sqrt{{(T_0-\tau)}^2 -{(|x_0|-|y|)}^2-4|x_0||y|\sin^2{\frac{\varphi}{2}}}} d \varphi
\end{eqnarray*}
Let $a=4|x_0||y|$, $b={(T_0-\tau)}^2 -{(|x_0|-|y|)}^2$, $z=\sin{\frac{\varphi}{2}}$, then we have
\begin{displaymath}
\left|\int_{\sin^2{\frac{\varphi}{2}} \leq \frac{{(T_0-\tau)}^2-{(|x_0|-|y|)}^2}{4|x_0||y|}} \frac{1}{\sqrt{{(T_0-\tau)}^2 -{(|x_0|-|y|)}^2-4|x_0||y|\sin^2{\frac{\varphi}{2}}}} d \varphi\right| = \left| \int_{0}^{\sqrt{\frac{b}{a}}} \frac{1}{\sqrt{b-az^2}} \frac{1}{\sqrt{1-z^2}} dz\right|
\end{displaymath}
\begin{eqnarray*}
\left| \int_{0}^{\sqrt{\frac{b}{a}}} \frac{1}{\sqrt{b-az^2}} \frac{1}{\sqrt{1-z^2}} dz\right| & \lesssim& \frac{1}{b^{\frac{1}{4}}} \int_{0}^{\sqrt{\frac{b}{a}}} \frac{1}{\sqrt{\sqrt{b}-\sqrt{a}z}} \frac{1}{\sqrt{1-z}} dz \\
& \lesssim & \frac{1}{a^{\frac{1}{4}}b^{\frac{1}{4}}} \int_{0}^{\sqrt{\frac{b}{a}}} \frac{1}{\sqrt{\sqrt{\frac{b}{a}}-z}} \frac{1}{\sqrt{1-\sqrt{\frac{b}{a}}+\sqrt{\frac{b}{a}}-z}} dz \\
& \lesssim & \frac{1}{a^{\frac{1}{4}}b^{\frac{1}{4}}} \int_{0}^{\sqrt{\frac{b}{a}}} \frac{1}{\sqrt{x}} \frac{1}{\sqrt{1-\sqrt{\frac{b}{a}}+x}} dx \\
& \lesssim & \frac{1}{a^{\frac{1}{4}}b^{\frac{1}{4}}{\left(1-\sqrt{\frac{b}{a}} \right)}^{\frac{1}{4}}} \int_{0}^{\sqrt{\frac{b}{a}}} \frac{1}{{x}^{\frac{3}{4}}} dx \\
& \lesssim & \frac{{\left( \sqrt{a}+\sqrt{b} \right)}^{\frac{1}{4}}}{b^{\frac{1}{8}} {(a-b)}^{\frac{1}{4}}} \\
& \lesssim & \frac{1}{b^{\frac{1}{8}} {(a-b)}^{\frac{1}{4}}}.
\end{eqnarray*}
Therefore, using $(\ref{5.15})$ , we have
\begin{eqnarray*}
{\uppercase\expandafter{\romannumeral3}}_{B1}^{(1)} & \lesssim & \int_{T_0-|x_0|}^{T_0+|x_0|} \int_{-v}^{T_0-|x_0|} \frac{{\phi}^2 r}{{(v+|x_0|-T_0)}^{\frac{1}{4}} {(T_0+|x_0|-u)}^{\frac{1}{4}} {(T_0-|x_0|-u)}^{\frac{1}{8}}{(T_0+|x_0|-v)}^{\frac{1}{8}}} du dv \\
& \lesssim & \int_{T_0-|x_0|}^{T_0+|x_0|} \frac{1}{{(v+|x_0|-T_0)}^{\frac{1}{4}} {(T_0+|x_0|-v)}^{\frac{1}{8}}}  dv
{\left( \int_{-v}^{T_0-|x_0|} \frac{r}{{(T_0+|x_0|-u)}^{\frac{1}{2}} {(T_0-|x_0|-u)}^{\frac{1}{4}}} du \right)}^{\frac{1}{2}} \\
& \lesssim & \int_{T_0-|x_0|}^{T_0+|x_0|} \frac{1}{{(v+|x_0|-T_0)}^{\frac{1}{4}} {(T_0+|x_0|-v)}^{\frac{1}{8}}}  dv
{\left( \int_{-v}^{T_0-|x_0|} \frac{1}{r^{2\delta-1}{(T_0+|x_0|-u)}^{\frac{1}{2}} {(T_0-|x_0|-u)}^{\frac{1}{4}}} du \right)}^{\frac{1}{2}} \\
& \lesssim &  \int_{T_0-|x_0|}^{T_0+|x_0|} \frac{1}{{(v+|x_0|-T_0)}^{\frac{1}{4}}{(T_0+|x_0|-v)}^{\frac{1}{8}}} dv \\
 &&{\left[ {\left( \int_{-v}^v \frac{1}{r^{4\delta-2} (T_0+|x_0|-u)} du \right)}^{\frac{1}{2}} {\left( \int_{-v}^{T_0-|x_0|} \frac{1}{\sqrt{T_0-|x_0|-u}} du \right)}^{\frac{1}{2}} \right]}^{\frac{1}{2}}.
\end{eqnarray*}
Then, by Lemma \ref{prop5.2},
\begin{eqnarray*}
{\uppercase\expandafter{\romannumeral3}}_{B1}^{(1)} & \lesssim &  \int_{T_0-|x_0|}^{T_0+|x_0|} \frac{1}{{(v+|x_0|-T_0)}^{\frac{1}{4}}{(T_0+|x_0|-v)}^{\frac{1}{8}}} {\left( \frac{{(T_0-|x_0|+v)}^{\frac{1}{4}}}{{(T_0+|x_0|-v)}^{2\delta-1}} \right)}^{\frac{1}{2}}dv  \\
& \lesssim &  \int_{T_0-|x_0|}^{T_0+|x_0|} \frac{{(T_0-|x_0|+v)}^{\frac{1}{8}}}{{(v+|x_0|-T_0)}^{\frac{1}{4}}{(T_0+|x_0|-v)}^{\delta-\frac{3}{8}}} dv \\
& \lesssim &  {\left( \int_{T_0-|x_0|}^{T_0+|x_0|} \frac{{(T_0-|x_0|+v)}^{\frac{1}{4}} {(T_0+|x_0|-v)}^{\frac{1}{4}}}{{(v+|x_0|-T_0)}^{\frac{1}{2}}} dv \right)}^{\frac{1}{2}} {\left( \int_{T_0-|x_0|}^{T_0+|x_0|} \frac{1}{{(T_0+|x_0|-v)}^{2\delta-\frac{1}{2}}} dv \right)}^{\frac{1}{2}} \\
& \lesssim &  {|x_0|}^{\frac{1}{8}} {\left( \int_{T_0-|x_0|}^{T_0+|x_0|} \frac{1}{{(v+|x_0|-T_0)}^{\frac{1}{2}}} dv \right)}^{\frac{1}{2}} {\left( \int_{T_0-|x_0|}^{T_0+|x_0|} \frac{1}{{(T_0+|x_0|-v)}^{2\delta-\frac{1}{2}}} dv \right)}^{\frac{1}{2}}.
\end{eqnarray*}
Thus,
\begin{equation*}
{\uppercase\expandafter{\romannumeral3}}_{B1}^{(1)}   \lesssim   {|x_0|}^{\frac{3}{8}} {|x_0|}^{\frac{3}{4}-\delta}
 \lesssim  r^{\frac{9}{8}-\delta}(u_0,v_0)
 \lesssim   r^{1-\delta}(u_0,v_0).
\end{equation*}
We can obtain estimates for ${\uppercase\expandafter{\romannumeral3}}_{B2}^{(1)} $ in a similar way.

Then, we have
\begin{eqnarray*}
{\uppercase\expandafter{\romannumeral3}}_{B3}^{(1)} &\lesssim&  X \int_{T_0-|x_0|}^{T_0+|x_0|}  \frac{1}{{(v+|x_0|-T_0)}^{\frac{1}{4}} {(T_0+|x_0|-v)}^{\frac{1}{8}}}dv \int_{-v}^{T_0-|x_0|} \frac{r^{1-\delta}}{ {(T_0+|x_0|-u)}^{\frac{1}{4}} {(T_0-|x_0|-u)}^{\frac{1}{8}}} du\\
\end{eqnarray*}
noting that
\begin{eqnarray*}
\int_{-v}^{T_0-|x_0|} \frac{r^{1-\delta}}{ {(T_0+|x_0|-u)}^{\frac{1}{4}} {(T_0-|x_0|-u)}^{\frac{1}{8}}} du & \lesssim & {\left( \int_{-v}^{T_0-|x_0|} \frac{1}{  r^{2\delta-1}{(T_0+|x_0|-u)}^{\frac{1}{2}}{(T_0-|x_0|-u)}^{\frac{1}{4}} } du \right)}^{\frac{1}{2}}\\
&&{\left( \int_{-v}^{T_0-|x_0|} R du \right)}^{\frac{1}{2}}\\
&\lesssim & {\left( \int_{-v}^{T_0-|x_0|} \frac{1}{  r^{2\delta-1}{(T_0+|x_0|-u)}^{\frac{1}{2}}{(T_0-|x_0|-u)}^{\frac{1}{4}} } du \right)}^{\frac{1}{2}} \\
&&(v-T_0+|x_0|)\\
& \lesssim &\varepsilon{\left( \int_{-v}^{T_0-|x_0|} \frac{1}{  r^{2\delta-1}{(T_0+|x_0|-u)}^{\frac{1}{2}}{(T_0-|x_0|-u)}^{\frac{1}{4}} } du \right)}^{\frac{1}{2}},
\end{eqnarray*}
thus, we can get
\begin{equation*}
{\uppercase\expandafter{\romannumeral3}}_{B3}^{(1)} \lesssim \varepsilon X r^{1-\delta}(u_0,v_0).
\end{equation*}

Finally, we have
\begin{eqnarray*}
{\uppercase\expandafter{\romannumeral3}}_{B4}^{(1)} & \lesssim&  \int_{T_0-|x_0|}^{T_0+|x_0|} \int_{-v}^{T_0-|x_0|} \frac{{|U_u|} r}{{(v+|x_0|-T_0)}^{\frac{1}{4}} {(T_0+|x_0|-u)}^{\frac{1}{4}} {(T_0-|x_0|-u)}^{\frac{1}{8}}{(T_0+|x_0|-v)}^{\frac{1}{8}}} du dv\\
&\lesssim &\varepsilon \int_{T_0-|x_0|}^{T_0+|x_0|}  \frac{1}{{(v+|x_0|-T_0)}^{\frac{1}{4}} {(T_0+|x_0|-v)}^{\frac{1}{8}}}  dv \\
&&{\left( \int_{-v}^{T_0-|x_0|} \frac{ r} {{(T_0+|x_0|-u)}^{\frac{1}{2}} {(T_0-|x_0|-u)}^{\frac{1}{4}}} du \right)}^{\frac{1}{2}} \\
& \lesssim & \int_{T_0-|x_0|}^{T_0+|x_0|} \frac{1}{{(v+|x_0|-T_0)}^{\frac{1}{4}} {(T_0+|x_0|-v)}^{\frac{1}{8}}}  dv\\
&&{\left( \int_{-v}^{T_0-|x_0|} \frac{1}{r^{2\delta-1}{(T_0+|x_0|-u)}^{\frac{1}{2}} {(T_0-|x_0|-u)}^{\frac{1}{4}}} du \right)}^{\frac{1}{2}} \\
&\lesssim& r^{1-\delta}(u_0,v_0).
\end{eqnarray*}
This finishes the proof of Lemma \ref{lem5.3}.
\end{proof}

With Lemma \ref{lem5.3}, we can estimate the quantity $X$ now.

Multiplying $(\ref{1.18})$ by $r^{\frac{1}{2}}$, we get
\begin{displaymath}
-4\partial_u \left( r^{\frac{1}{2}} U_v \right)=r^{\frac{1}{2}}\mathcal{H}+2\frac{r_v}{r^{\frac{1}{2}}} U_u,
\end{displaymath}
integrating the above equation with respect to $u$, we obtain
\begin{eqnarray*}
 -4r^{\frac{1}{2}}(v_0-|x_0|,|x_0|) U_v (v_0-|x_0|,|x_0|)&=&-4r^{\frac{1}{2}}(0,2v_0) U_v(0,2v_0) \\
 &&+ \int_{-v_0}^{v_0-2|x_0|} r^{\frac{1}{2}}e^{2\lambda} \mathcal{H} du + 2\int_{-v_0}^{v_0-2|x_0|} \frac{r_v}{r^{\frac{1}{2}}} U_u du \\
& \triangleq & \mathcal{A}_1+\mathcal{A}_{\mathcal{H}}+\mathcal{A}_3.
\end{eqnarray*}

By the regularity of the initial data,
\begin{equation*}
|\mathcal{A}_1| \leq A{r(u_0,v_0)}^{\frac{1}{2}-\delta}.
\end{equation*}

For $\mathcal{A}_3$, we have
\begin{eqnarray*}
\mathcal{A}_3 &=&  \int_{-v_0}^{v_0-2|x_0|} \frac{r_v}{r^{\frac{1}{2}}} {\left(U(u,v_0)-U(v_0,v_0)\right)}_u du         \\
&=&  \frac{r_v}{r^{\frac{1}{2}}} {\left(U(v_0-|x_0|,v_0)-U(v_0,v_0)\right)}-\frac{r_v}{r^{\frac{1}{2}}} {\left(U(-v_0,v_0)-U(v_0,v_0)\right)} \\
&&+ \frac{1}{2}\int_{-v_0}^{v_0-2|x_0|} \frac{r_v r_u}{r^{\frac{3}{2}}} {\left(U(u,v_0)-U(v_0,v_0)\right)}du\\
&&-\int_{-v_0}^{v_0-2|x_0|} \frac{r_{uv}}{r^{\frac{1}{2}}} {\left(U(u,v_0)-U(v_0,v_0)\right)} du    \\
&=& \mathcal{A}_3^{(1)}+\mathcal{A}_3^{(2)}+\mathcal{A}_3^{(3)}+\mathcal{A}_3^{(4)}.
\end{eqnarray*}

For $\mathcal{A}_3^{(1)}$ and $\mathcal{A}_3^{(2)}$, by Lemma \ref{lem5.3}, we have the following estimate,
\begin{displaymath}
|\mathcal{A}_3^{(1)}| \lesssim {|x_0|}^{\frac{1}{2}-\delta} |A+\varepsilon X|,
\end{displaymath}
and
\begin{displaymath}
|\mathcal{A}_3^{(2)}| \lesssim {(v_0)}^{\frac{1}{2}-\delta}|A+\varepsilon X| \lesssim   {|x_0|}^{\frac{1}{2}-\delta}|A+\varepsilon X|.
\end{displaymath}
Similarly for $\mathcal{A}_3^{(3)}$, we obtain
\begin{eqnarray*}
|\mathcal{A}_3^{(3)}| &\lesssim& \left( A+\varepsilon X \right) \int_{-v_0}^{v_0-2|x_0|} \frac{{(v-u)}^{1-\delta}}{{(v-u)}^{\frac{3}{2}}} du \\
& \lesssim & {|x_0|}^{\frac{1}{2}-\delta} \left( A+\varepsilon X \right).
\end{eqnarray*}

For $\mathcal{A}_3^{(4)}$, using \eqref{1.13} and \eqref{5.15}, we have
\begin{eqnarray*}
|\mathcal{A}_3^{(4)}| & \lesssim &  \int_{-v_0}^{v_0-2|x_0|} \phi^2 r^{\frac{1}{2}} {\left(U(u,v_0)-U(v_0,v_0)\right)} du  \\
& \lesssim & (A+\varepsilon X) \int_{-v_0}^{v_0-2|x_0|} \phi^2 r^{\frac{3}{2}-\delta} du \\
& \lesssim & (A+\varepsilon X) {r(u_0,v_0)}^{\frac{1}{2}-\delta},
\end{eqnarray*}
thus, we obtain
\begin{displaymath}
|\mathcal{A}_3| \lesssim (A+\varepsilon X) {r(u_0,v_0)}^{\frac{1}{2}-\delta}.
\end{displaymath}

For $\mathcal{A}_{\mathcal{H}_1}$ and $\mathcal{A}_{\mathcal{H}_2}$, using $(\ref{5.15})$, we have
\begin{eqnarray*}
|\mathcal{A}_{\mathcal{H}_1}| & \leq & C {\left( \int_{-v_0}^{u_0} R{\phi}^{4} du \right)}^{\frac{1}{2}} \\
 & \leq & C {|x_0|}^{\frac{1}{2}-\delta} {\left( \int_{-v_0}^{u_0} {\phi}^{4} R du \right)}^{\frac{1}{2}}  \\
 &=& C {|x_0|}^{\frac{1}{2}-\delta} {\left( \int_{|x_0|}^{v_0} {\phi}^{4} R dR \right)}^{\frac{1}{2}} \\
 & \lesssim & {|x_0|}^{\frac{1}{2}-\delta} \lesssim  r^{\frac{1}{2}-\delta}(u_0,v_0).
\end{eqnarray*}
We can get estimates for $\mathcal{A}_{\mathcal{H}_2}$ in a same way.

Therefore,
\begin{displaymath}
{r}^{\frac{1}{2}}(u_0,v_0) \left| U_v(u_0,v_0) \right| \lesssim  (A+\varepsilon X) {r}^{\frac{1}{2}-\delta}(u_0,v_0),
\end{displaymath}
then,
\begin{displaymath}
{r}^{\delta}(u_0,v_0) \left| U_v(u_0,v_0) \right| \lesssim  (A+\varepsilon X),
\end{displaymath}
thus, we obtain
\begin{displaymath}
X \lesssim  (A+\varepsilon X),
\end{displaymath}
which implies
\begin{equation} \label{5.17}
X \leq A.
\end{equation}

Now we give a uniform upper bound of $U$.
\begin{lemma} \label{lem5.4}
There exists a constant $A>0$ such that
\begin{equation} \label{5.18}
|U| \leq A
\end{equation}
holds in $K_0$.
\end{lemma}
\begin{proof}
Let $(u,v) \in K_0$, by \eqref{5.17} and Lemma \ref{lem5.3}, we have
\begin{displaymath}
|U(u,v)-U(v,v)| \lesssim A R^{1-\delta},
\end{displaymath}
then if $u \geq -u$, by \eqref{5.17}, we have
\begin{displaymath}
|U_v(u,v)| \lesssim A R^{-\delta},
\end{displaymath}
integrating the above inequality yields
\begin{displaymath}
|U(u,v)-U(u,u)| \lesssim A R^{1-\delta},
\end{displaymath}
adding the above two inequality, we get
\begin{displaymath}
|U(v,v)-U(u,u)| \lesssim A R^{1-\delta}.
\end{displaymath}
Now take $u=0$, then by the regularity of the initial data, we get a uniform bound on $\Gamma$,
\begin{displaymath}
|U(v,v)| \lesssim |U(0,0)|+R^{1-\delta} \leq A.
\end{displaymath}
Then for any $(u,v) \in K_0$, by \eqref{5.17}, we have
\begin{displaymath}
|U(u,v)| \lesssim |U(u,u)|+R^{1-\delta} \leq A\quad if \,\, u\geq -u,
\end{displaymath}
and
\begin{displaymath}
|U(u,v)| \lesssim |U(u,-u)|+R^{1-\delta}  \leq A\quad if \,\, u< -u.
\end{displaymath}
This finishes the proof of Lemma \ref{lem5.4}.
\end{proof}

Now, taking advantage of Lemma \ref{lem5.4}, we give the following lemma.
\begin{lemma} \label{lem5.5}
For any $1\leq s <2$, we have
\begin{displaymath}
\left| \partial_v r-\frac{1}{2} \right| \lesssim \varepsilon R^s,\,\left| \partial_u r+\frac{1}{2} \right| \lesssim \varepsilon R^s,\,|r-R| \lesssim \varepsilon R^{1+s}
\end{displaymath}
in a small cone $K \subset K_0$.
\end{lemma}
\begin{proof}
Still, we integrate \eqref{1.13} from the axis of symmetry $\Gamma$ along the null direction and in view of \cite{chen} and the initialization on $\Gamma$, we deduce
\begin{eqnarray*}
\left|\partial_u r+\frac{1}{2} \right| & \lesssim & {\left(\int_{-v}^v \phi^{2p} R du\right)}^{\frac{1}{p}} R^{2-\frac{2}{p}} \\
& \lesssim & {\left(\int_{0}^{v} {\phi}^{2p}(v-2R,v) R dR\right)}^{\frac{1}{p}} R^{2-\frac{2}{p}}
\end{eqnarray*}
for some $p>1$.

Using the U(1) symmetry, Lemma \ref{lem5.4}, energy estimates on the flux and Sobolev's inequality, we get
\begin{eqnarray*}
{\left( \int_{0}^{v} {\phi}^{2p}(v-2R,v) R dR \right)}^{\frac{1}{2p}} &=& {\left( \frac{1}{2\pi} \int_{0}^{2\pi} \int_{0}^{v} {\phi}^{2p}(v-2R,v) R dR d\theta \right)}^{\frac{1}{2p}} \\ \nonumber
& \lesssim & {\left( \int_{-v}^{v} {\phi_u}^{2} R du \right)}^{\frac{1}{2}}+{\left( \int_{-v}^{v} {\phi}^{2} R du \right)}^{\frac{1}{2}} \\ \nonumber
& \lesssim & \varepsilon.
\end{eqnarray*}

Let $s\triangleq2-\frac{2}{p}$, we have
\begin{displaymath}
|\partial_u r+\frac{1}{2}| \lesssim \varepsilon R^{s}.
\end{displaymath}
Similarly,
\begin{displaymath}
|\partial_v r-\frac{1}{2}| \lesssim \varepsilon R^{s}.
\end{displaymath}
Then, integrating along another null direction yields
\begin{displaymath}
|r-R| \lesssim \varepsilon R^{1+s}.
\end{displaymath}

This concludes the proof of the lemma.
\end{proof}

\subsection{Higher regularity estimates}
Now we differentiate equation \eqref{1.18} with respect to $R$ and denote $V=\partial_R U$. Then we obtain
\begin{equation} \label{5.19}
-V_{TT}+V_{RR}+\frac{V_R}{R}-\frac{V}{R^2}=F,
\end{equation}
where
\begin{eqnarray*}
F&= & \partial_R \mathcal{H}+ \left( \frac{2r_u}{r}+\frac{1}{R} \right) V_v +\left(- \frac{1}{R^2}-\frac{2r_u}{r^2} (r_v-r_u)\right) U_v+\frac{2r_{uR}}{r} U_v \\
&& +\left( \frac{2r_v}{r}-\frac{1}{R} \right) V_u +\left( \frac{1}{R^2}-\frac{2r_v}{r^2} (r_v-r_u)\right) U_u+\frac{2r_{vR}}{r} U_u \\
& \triangleq & F_{\mathcal{H}}+F_7+F_8+F_9+F_{10}+F_{11}+F_{12}.
\end{eqnarray*}
Calculating $\partial_R \mathcal{H}_1$ and $\partial_R \mathcal{H}_2$ respectively, the following six terms would appear,
\begin{displaymath}
F_1 \triangleq  m^2e^{2\lambda-2\gamma}\gamma_R{\phi}^2,
\end{displaymath}
\begin{displaymath}
F_2 \triangleq -m^2e^{2\lambda-2\gamma} {\phi}\phi_R,
\end{displaymath}
\begin{displaymath}
F_3 \triangleq -2m^2e^{2\lambda-2\gamma}\gamma_R \phi,
\end{displaymath}
\begin{displaymath}
F_4 \triangleq  m^2e^{2\lambda-2\gamma} \phi_R,
\end{displaymath}
\begin{displaymath}
F_5 \triangleq -m^2e^{2\lambda-2\gamma} \lambda_R{\phi}^2 ,
\end{displaymath}
\begin{displaymath}
F_6 \triangleq 2m^2e^{2\lambda-2\gamma}\lambda_R\phi.
\end{displaymath}

Let us define
\begin{displaymath}
X^{(0)}=\sup\limits_{(T,|x|) \in K_0} r^{\delta} \left| V_v(T,|x|) \right|.
\end{displaymath}
Our goal is to show the boundness of $X^{(0)}$, it is obviously bounded by a constant $M$ on the part $\{0<T<T_1\}$ by regularity where $0<T^*-T_1<<T^*-\bar{T}$. Thus, we only need to consider on
\begin{displaymath}
K_1=\left\{ (T,R)| 0 \leq R <T^*-T,\ T \geq T_1 \right\}.
\end{displaymath}

Let us introduce
\begin{displaymath}
X^{(1)}=\sup\limits_{(T,|x|) \in K_1} r^{\delta} \left| V_v(T,|x|) \right|,
\end{displaymath}
\begin{displaymath}
Y^{(0)}=\sup\limits_{(T,|x|) \in K_0} r^{\delta-1} \left| V(T,|x|) \right|,
\end{displaymath}
\begin{displaymath}
Y^{(1)}=\sup\limits_{(T,|x|) \in K_1} r^{\delta-1} \left| V(T,|x|) \right|,
\end{displaymath}
\begin{displaymath}
Z^{(0)}=\sup\limits_{(T,|x|) \in K_0}{\left( \int_{-v}^{T-|x|} {V_u}^2r du \right)}^{\frac{1}{2}},
\end{displaymath}
\begin{displaymath}
Z^{(1)}=\sup\limits_{(T,|x|) \in K_1}{\left( \int_{2T_1-v}^{T-|x|} {V_u}^2r du \right)}^{\frac{1}{2}}.
\end{displaymath}
and
\begin{displaymath}
P\triangleq  r_{uR} ,\quad Q\triangleq  r_{vR},\quad W\triangleq  \lambda_{R},
\end{displaymath}
\begin{displaymath}
L^{(0)} \triangleq \sup\limits_{(T,|x|) \in K_0} |r^{\delta-1}W(T,|x|)|, \quad L^{(1)} \triangleq \sup\limits_{(T,|x|) \in K_1} |r^{\delta-1}W(T,|x|)|.
\end{displaymath}

Obviously, there holds
\begin{equation} \label{5.20}
X^{(0)} \leq X^{(1)}+M
\end{equation}
where $M$ is a constant depending on $K_1$.

Noting that $V(T,0)=0$, we aim to get a similar estimate as we did in Lemma \ref{lem5.3}. Firstly, we have
\begin{eqnarray*}
|V| & \lesssim & X^{(0)} \int_u^v {r'}^{-\delta} dv' \\
& \lesssim & X^{(0)} r^{1-\delta},
\end{eqnarray*}
therefore
\begin{equation} \label{5.21}
Y^{(0)} \lesssim X^{(0)}.
\end{equation}

Then, for $Z$, we have
\begin{equation} \label{5.22}
Z^{(0)} \leq Z^{(1)}+M.
\end{equation}

Before estimating $L^{(0)}$ and $L^{(1)}$, we give the following lemma.
\begin{lemma} \label{remark3.10}
The following normalizations hold on $\Gamma$,
\begin{displaymath}
R=0:\,\partial \lambda=0,\, \partial^2 r=0
\end{displaymath}
\end{lemma}
\begin{proof}
By \cite{chen}, we know that $\lambda=0$ on $\Gamma$, thus,
\begin{displaymath}
\lambda_T=0,
\end{displaymath}
and it is obvious that $\lambda_R=0$ by symmetry, so we have $\partial \lambda=0$ on $\Gamma$.

By \eqref{1.13}, there holds $r_{uv}=0$ on $\Gamma$. Then noting that $\partial \lambda=0$ and using \eqref{1.12}, we have
\begin{eqnarray*}
r_{uu}&=&-r(2{\gamma_u}^2+{\phi_u}^2)+2\lambda_u \cdot r_u \\
&=& 2\lambda_u \cdot r_u \\
&=& 0.
\end{eqnarray*}
Similarly, we can get $$r_{vv}=0.$$
\end{proof}
Thus, using Lemma \ref{remark3.10}, we have
\begin{displaymath}
\left| \frac{\lambda_R(T_0,R_0)}{R_0} \right| \leq M, \quad \forall (T_0,R_0) \in K_0 \setminus K_1,
\end{displaymath}
which implies
\begin{displaymath}
L^{(0)} \leq L^{(1)}+M.
\end{displaymath}

Then we estimate $W$.  By differentiating \eqref{1.15} with respect to $R$, we have
\begin{eqnarray*}
\partial_u \partial_v W&=&-\gamma_{uR}\cdot \gamma_v-\gamma_u \cdot \gamma_{vR}-\frac{1}{2} \phi_{uR}\cdot \phi_v-\frac{1}{2}\phi_u \cdot \phi_{vR}\\
&&+\frac{m^2}{4} e^{2\lambda-2\gamma} (W-V) \cdot \phi^2+\frac{m^2}{4} e^{2\lambda-2\gamma} \phi \cdot V,
\end{eqnarray*}
integrating the above equality, we obtain
\begin{eqnarray} \label{5.23}
 \nonumber |W_v(u,v)| & \lesssim & |\partial_v W(-v,v)|+\int_{-v}^u |V_u|\cdot |U_v| d u'+\int_{-v}^u |V_v|\cdot |U_u| d u' \\
&&+\int_{-v}^u |W|\cdot \phi^2 d u'+\int_{-v}^u |V|\cdot \phi^2 d u'+\int_{-v}^u |V|\cdot \phi d u' \\ \nonumber
&\lesssim & A+Z^{(0)}{\left( \int_{-v}^u R^{-1-2\delta} du' \right)}^{\frac{1}{2}}+\varepsilon X^{(0)}{\left( \int_{-v}^u R^{-1-2\delta} du' \right)}^{\frac{1}{2}}\\ \nonumber
&&+(Y^{(0)}+L^{(0)}) \int_{-v}^u  \phi^2 R^{1-\delta}d u'+Y^{(0)}\int_{-v}^u  \phi R^{1-\delta}d u' \\ \nonumber
& \lesssim & R^{-\delta}(A+X^{(0)}+Y^{(0)}+Z^{(0)}+\varepsilon L^{(0)})+\varepsilon Y^{(0)}{\left(\int_{-v}^u   R^{1-2\delta}d u'\right)}^{\frac{1}{2}} \\ \nonumber
& \lesssim & R^{-\delta}(A+X^{(0)}+Y^{(0)}+Z^{(0)}+\varepsilon L^{(0)}) +\varepsilon R^{-\delta}Y^{(0)}{\left(\int_{-v}^u   Rd u'\right)}^{\frac{1}{2}} \\ \nonumber
& \lesssim & R^{-\delta}(A+X^{(0)}+Y^{(0)}+Z^{(0)}+\varepsilon L^{(0)})
\end{eqnarray}
for any $(u,v) \in K_0$.

Then integrating \eqref{5.23} with respect to $v$, for any $(u,v) \in K_1$, we have
\begin{eqnarray*}
|W| &\lesssim& R^{1-\delta}(A+X^{(0)}+Y^{(0)}+Z^{(0)}+\varepsilon L^{(0)}) \\
&\lesssim& R^{1-\delta}(M+X^{(0)}+Y^{(0)}+Z^{(0)}+\varepsilon L^{(1)}),
\end{eqnarray*}
which implies
\begin{equation} \label{5.24}
L^{(1)} \lesssim M+X^{(0)}+Y^{(0)}+Z^{(0)}.
\end{equation}

Now we differentiate \eqref{1.13} with respect to $R$, we have
\begin{displaymath}
P_v =Q_u= \frac{m^2}{4} r_R e^{2\lambda-2\gamma} \phi^2+\frac{m^2}{2} r e^{2\lambda-2\gamma}(W-V) \phi^2+\frac{m^2}{2}r e^{2\lambda-2\gamma} \phi \cdot V,
\end{displaymath}
integrating the above equality with respect to $u$, using \eqref{5.24}, Lemma \ref{remark3.10} and Lemma \ref{lem5.4}, we can get on $K_1$ that
\begin{eqnarray*}
|Q(u,v)| & \lesssim & \int_{u}^v |r_R| \phi^2 du'+\int_{u}^v |W-V| \phi^2 r du'+Y^{(1)}\int_{u}^v |\phi| r^{2-\delta} du' \\
& \lesssim & R+(M+X^{(0)}+Y^{(0)}+Z^{(0)})\int_{u}^v \phi^2 r^{2-\delta} du'+Y^{(1)}\int_{u}^v \phi r^{2-\delta} du'\\
& \lesssim & eR^{\frac{3}{2}-\delta}+\varepsilon (M+X^{(0)}+Y^{(0)}+Z^{(0)}){\left(\int_{u}^v R^{3-2\delta} du'\right)}^{\frac{1}{2}}\\
& \lesssim & e(M+X^{(0)}+Y^{(0)}+Z^{(0)})R^{\frac{3}{2}-\delta}.
\end{eqnarray*}
Similarly, integration with respect to $v$ yields
\begin{eqnarray*}
|P(u,v)| & \lesssim & \int_{u}^v |r_R| \phi^2 dv'+\int_{u}^v |W-V| \phi^2 r dv'+Y^{(0)}\int_{u}^v |\phi| r^{2-\delta} dv' \\
& \lesssim & R+(M+X^{(0)}+Y^{(0)}+Z^{(0)})\int_{u}^v \phi^2 r^{2-\delta} dv'+Y^{(0)}\int_{u}^v \phi r^{2-\delta} dv'\\
& \lesssim & eR^{\frac{3}{2}-\delta}+e(M+X^{(0)}+Y^{(0)}+Z^{(0)}){\left(\int_{u}^v R^{3-2\delta} dv'\right)}^{\frac{1}{2}}\\
& \lesssim & e(M+X^{(0)}+Y^{(0)}+Z^{(0)})R^{\frac{3}{2}-\delta}.
\end{eqnarray*}

Now we give an estimate for $Z^{(1)}$. By shifting time, we may take $T_1=0$, then, by energy estimates of \eqref{5.19}, we have
\begin{eqnarray*}
 \int_{-T-|x|}^{T-|x|} {{V}_u}^2(u,T+|x|)R du
 &\lesssim &  \int_{-T-|x|}^{T+|x|} {V_{v}}^2(-T-|x|,v)R d v+ \left|\int_{-T-|x|}^{T} \int_{0}^{T+|x|-|\tau|} V_{\tau} \cdot  F R dR d \tau\right| \\
 & \lesssim & M+\int_{T}^{T+|x|} \int_{-v}^{T}  |V_{\tau}| \cdot |F| R du dv +\int_{0}^{T+|x|} \int_{-T-|x|}^{-v} |V_{\tau}| \cdot |F| R du dv  \\
 &&+\int_{0}^{T} \int_{-v}^{v}  |V_{\tau}| \cdot |F| R du dv +\int_{-T-|x|}^0 \int_{-T-|x|}^{v}  |V_{\tau}| \cdot |F|R du dv \\
 &\triangleq & M+\alpha+\beta+\eta+\mu
\end{eqnarray*}

With the use of the regularity, we can easily get estimates for $\beta$ and $\mu$,
\begin{displaymath}
|\beta| \leq M, \quad |\mu| \leq M.
\end{displaymath}

Then, we can define $\alpha_j$ and $\eta_j$ respectively as in \eqref{5.19}, and we define $\alpha_j^u$ and $\alpha_j^v$ as follows,
\begin{displaymath}
\alpha_j^u \triangleq \int_{T}^{T+|x|} \int_{-v}^{T}  |V_{u}| \cdot |F_j| R du dv
\end{displaymath}
\begin{displaymath}
\alpha_j^v \triangleq \int_{T}^{T+|x|} \int_{-v}^{T}  |V_{v}| \cdot |F_j| R du dv
\end{displaymath}
where $j=1,\cdots,12$. Similarly, we can define $\eta_j^u,\eta_j^v$.

For ${\alpha}_1$, we have
\begin{eqnarray*}
{\alpha}_1^u  &\lesssim &    \int_{T}^{T+|x|} \int_{-v}^{T} |V_u|\cdot |V| \cdot{|\phi|}^2 R du dv  \\
 & \lesssim & Y^{(1)}  \int_{T}^{T+|x|} \int_{-v}^{T}  |V_u| \cdot{|\phi|}^2 R^{2-\delta} du dv \\
 & \lesssim &  Y^{(1)}Z^{(1)} \int_{T}^{T+|x|}dv{\left(   \int_{-v}^{T}  {|\phi|}^4 R^{3-2\delta} du \right)}^{\frac{1}{2}} \\
 & \lesssim & e Y^{(1)}Z^{(1)},
\end{eqnarray*}
and
\begin{eqnarray*}
{\alpha}_1^v &\lesssim &   X^{(1)} Y^{(1)}  \int_{T}^{T+|x|} \int_{-v}^{T} {|\phi|}^2 R^{2-2\delta} du dv \\
 & \lesssim &  X^{(1)} Y^{(1)} \int_{T}^{T+|x|} dv {\left(   \int_{-v}^{T}  {|\phi|}^4 R du \right)}^{\frac{1}{2}}{\left(   \int_{-v}^{T}   R^{3-4\delta} du \right)}^{\frac{1}{2}} \\
 & \lesssim & e X^{(1)} Y^{(1)}.
\end{eqnarray*}
Thus,
\begin{displaymath}
{\alpha}_1  \lesssim e Y^{(1)}Z^{(1)}+e X^{(1)} Y^{(1)},
\end{displaymath}
where $e(\varepsilon,T-T_1)$ are sufficiently small. And we can get similar estimates for ${\alpha}_2,{\alpha}_3,{\alpha}_4$.

Then, for ${\alpha}_{5}$, we have
\begin{eqnarray*}
{\alpha}_{5}^u  &\lesssim &    \int_{T}^{T+|x|} \int_{-v}^{T} |V_u|\cdot |W| \cdot{|\phi|}^2 R du dv   \\
 & \lesssim & L^{(1)}  \int_{T}^{T+|x|} \int_{-v}^{T}  |V_u| \cdot{|\phi|}^2 R^{2-\delta} du dv \\
 & \lesssim &  L^{(1)}Z^{(1)} \int_{T}^{T+|x|}dv{\left(   \int_{-v}^{T}  {|\phi|}^4 R^{3-2\delta} du \right)}^{\frac{1}{2}} \\
 & \lesssim & e (M+X^{(0)}+Y^{(0)}+Z^{(0)}) \cdot Z^{(1)} \\
  & \lesssim & e (M+{X^{(0)}}^2+{Y^{(0)}}^2+{Z^{(0)}}^2),
\end{eqnarray*}
and
\begin{eqnarray*}
{\alpha}_{5}^v  &\lesssim &   X^{(1)} L^{(1)}  \int_{T}^{T+|x|} \int_{-v}^{T} {|\phi|}^2 R^{2-2\delta} du dv   \\
 & \lesssim &  X^{(1)} L^{(1)} \int_{T}^{T+|x|} dv {\left(   \int_{-v}^{T}  {|\phi|}^4 R du \right)}^{\frac{1}{2}}{\left(   \int_{-v}^{T}   R^{3-4\delta} du \right)}^{\frac{1}{2}} \\
  & \lesssim & e (M+X^{(0)}+Y^{(0)}+Z^{(0)})\cdot X^{(1)} \\
  & \lesssim & e (M+{X^{(0)}}^2+{Y^{(0)}}^2+{Z^{(0)}}^2),
\end{eqnarray*}
thus,
\begin{displaymath}
{\alpha}_{5}  \lesssim e (M+{X^{(0)}}^2+{Y^{(0)}}^2+{Z^{(0)}}^2).
\end{displaymath}
Similar estimates hold for $\alpha_{6}$.

Using Lemma \ref{lem5.5}, we can get
$$|\frac{2R r_u+r}{r} | \lesssim \varepsilon R,$$
then we have
\begin{eqnarray*}
{\alpha}_7^v  &\lesssim &   {X^{(1)}}^2  \int_{T}^{T+|x|} \int_{-v}^{T}  R^{1-2\delta} du dv   \\
 & \lesssim & e {X^{(1)}}^2,
\end{eqnarray*}
and
\begin{eqnarray*}
{\alpha}_7^u  &\lesssim &   X^{(1)} Z^{(1)}  \int_{T}^{T+|x|} dv {\left(\int_{-v}^{T}  R^{1-2\delta} du \right)}^{\frac{1}{2}}   \\
 & \lesssim & e X^{(1)} Z^{(1)},
\end{eqnarray*}
thus,
\begin{displaymath}
{\alpha}_7  \lesssim e {X^{(1)}}^2+e X^{(1)} Z^{(1)}.
\end{displaymath}

Similarly, we get estimates for ${\alpha}_{10}$,
\begin{eqnarray*}
{\alpha}_{10}^v  &\lesssim &   X^{(1)} Z^{(1)}  \int_{T}^{T+|x|} dv {\left(\int_{-v}^{T}  R^{1-2\delta} du \right)}^{\frac{1}{2}}   \\
 & \lesssim & e X^{(1)} Z^{(1)},
\end{eqnarray*}
and
\begin{eqnarray*}
{\alpha}_{10}^u  &\lesssim &   {Z^{(1)}}^2  \int_{T}^{T+|x|} dv    \\
 & \lesssim & e  {Z^{(1)}}^2,
\end{eqnarray*}
thus,
\begin{displaymath}
{\alpha}_{10}  \lesssim e {Z^{(1)}}^2+e X^{(1)} Z^{(1)}.
\end{displaymath}

Now we give estimates for $\alpha_8$,
\begin{displaymath}
{\alpha}_8^v  \lesssim      \int_{T}^{T+|x|} \int_{-v}^{T}  |V_v| \cdot |1+\frac{2r_u R^2}{r^2} r_v-\frac{2r_u R^2}{r^2}r_u| \cdot |U_v| R^{-1}du dv.
\end{displaymath}
Noting that
\begin{eqnarray*}
|1+\frac{2r_u R^2}{r^2} r_v-\frac{2r_u R^2}{r^2}r_u|
& = &   |1+\frac{2r_u R^2}{r^2} (r_v-r_u-1)+\frac{2r_u R^2}{r^2}|  \\
& = &   |\frac{r^2-R^2+(2r_u+1)R^2}{r^2}+\frac{(2r_u+1) R^2}{r^2} (r_v-r_u-1)-\frac{R^2}{r^2}(r_v-r_u-1)|  \\
& \lesssim & \varepsilon R.
\end{eqnarray*}
Thus,
\begin{eqnarray*}
{\alpha}_8^v
& \lesssim & \varepsilon \int_{T}^{T+|x|} \int_{-v}^{T}  |V_v|  |U_v| du dv \\
& \lesssim &  e {X^{(1)}} \int_{T}^{T+|x|} \int_{-v}^{T}  R^{-2\delta}du dv \\
& \lesssim &  e {X^{(1)}}.
\end{eqnarray*}

Then, noting that by Lemma \ref{lem5.5},
\begin{displaymath}
|2 r_u+1 | \lesssim \varepsilon R^{1+\frac{1}{2}+\frac{1}{4}},\,|2r_v-1 | \lesssim \varepsilon R^{1+\frac{1}{2}+\frac{1}{4}},\,|r-R| \lesssim \varepsilon R^{2+\frac{1}{2}+\frac{1}{4}},
\end{displaymath}
then we have
\begin{eqnarray*}
{\alpha}_8^u  &\lesssim &     \int_{T}^{T+|x|} \int_{-v}^{T}  |V_u| \cdot |1+\frac{2r_u R^2}{r^2} r_v-\frac{2r_u R^2}{r^2}r_u| \cdot |U_v| R^{-1}du dv \\
& \lesssim & \varepsilon \int_{T}^{T+|x|} \int_{-v}^{T}  |V_u|  |U_v| R^{\frac{3}{4}}du dv \\
& \lesssim &  \varepsilon {Z^{(1)}} {\left(\int_{T}^{T+|x|} \int_{-v}^{T}  R^{\frac{1}{2}-2\delta}du dv\right)}^{\frac{1}{2}} \\
& \lesssim &  \varepsilon {Z^{(1)}},
\end{eqnarray*}
thus,
\begin{eqnarray*}
{\alpha}_8 &\lesssim &      e {X^{(1)}}+ e {Z^{(1)}} \\
&\lesssim &      e ({X^{(1)}}^2+  {Z^{(1)}}^2+1).
\end{eqnarray*}
Similarly, we can get
\begin{displaymath}
{\alpha}_{11}^v   \lesssim   e {X^{(1)}},
\end{displaymath}
and
\begin{eqnarray*}
{\alpha}_{11}^u  &\lesssim &   e {Z^{(1)}} {\left(\int_{T}^{T+|x|} \int_{-v}^{T}  R^{\frac{1}{2}} {U_u}^2 du dv\right)}^{\frac{1}{2}} \\
&\lesssim &   \varepsilon {Z^{(1)}} {\left(\int_{T}^{T+|x|} {(v-T)}^{-\frac{1}{2}} dv\right)}^{\frac{1}{2}} \\
& \lesssim &  \varepsilon {Z^{(1)}},
\end{eqnarray*}
thus,
\begin{eqnarray*}
{\alpha}_{11}  &\lesssim &      e {X^{(1)}}+ e {Z^{(1)}} \\
&\lesssim &      e ({X^{(1)}}^2+  {Z^{(1)}}^2+1).
\end{eqnarray*}

Now we estimate $\alpha_9$ and $\alpha_{12}$. For $\alpha_9$, we have
\begin{eqnarray*}
{\alpha}_9^u  &\lesssim &     \int_{T}^{T+|x|} \int_{-v}^{T}  |V_u| |P| |U_v| du dv \\
& \lesssim & e(M+X^{(0)}+Y^{(0)}+Z^{(0)}) \int_{T}^{T+|x|} \int_{-v}^{T}  |V_u|  |U_v| R^{\frac{3}{2}-\delta}du dv \\
& \lesssim &  e {Z^{(0)}}(M+X^{(0)}+Y^{(0)}+Z^{(0)}) {\left(\int_{T}^{T+|x|} \int_{-v}^{T}  R^{2-4\delta}du dv\right)}^{\frac{1}{2}} \\
& \lesssim &  e {Z^{(0)}}(M+X^{(0)}+Y^{(0)}+Z^{(0)}),
\end{eqnarray*}
and
\begin{eqnarray*}
{\alpha}_9^v  &\lesssim &     \int_{T}^{T+|x|} \int_{-v}^{T}  |V_v| |P| |U_v| du dv \\
& \lesssim & e(M+X^{(0)}+Y^{(0)}+Z^{(0)}) \int_{T}^{T+|x|} \int_{-v}^{T}  |V_v|  |U_v| R^{\frac{3}{2}-\delta}du dv \\
& \lesssim &  e {X^{(0)}}(M+X^{(0)}+Y^{(0)}+Z^{(0)}) \int_{T}^{T+|x|} \int_{-v}^{T}  R^{\frac{3}{2}-3\delta}du dv \\
& \lesssim &  e {X^{(0)}}(M+X^{(0)}+Y^{(0)}+Z^{(0)}),
\end{eqnarray*}
thus,
\begin{eqnarray*}
{\alpha}_9  &\lesssim &      e {X^{(0)}}(M+X^{(0)}+Y^{(0)}+Z^{(0)})+ e {Z^{(0)}}(M+X^{(0)}+Y^{(0)}+Z^{(0)}) \\
&\lesssim &      e ({X^{(0)}}^2+{Y^{(0)}}^2+{Z^{(0)}}^2+M).
\end{eqnarray*}

For $\alpha_{12}$, we have
\begin{eqnarray*}
{\alpha}_{12}^v  &\lesssim &     \int_{T}^{T+|x|} \int_{-v}^{T}  |V_v| |Q| |U_u| du dv \\
& \lesssim & e(M+X^{(0)}+Y^{(0)}+Z^{(0)}) \int_{T}^{T+|x|} \int_{-v}^{T}  |V_v|  |U_u| R^{\frac{3}{2}-\delta}du dv \\
& \lesssim &  e {X^{(0)}}(M+X^{(0)}+Y^{(0)}+Z^{(0)}) {\left(\int_{T}^{T+|x|} \int_{-v}^{T}  R^{2-4\delta}du dv\right)}^{\frac{1}{2}} \\
& \lesssim &  e {X^{(0)}}(M+X^{(0)}+Y^{(0)}+Z^{(0)}),
\end{eqnarray*}
and
\begin{eqnarray*}
{\alpha}_{12}^u &\lesssim &     \int_{T}^{T+|x|} \int_{-v}^{T}  |V_u| |Q| |U_u| du dv \\
& \lesssim & e(M+X^{(0)}+Y^{(0)}+Z^{(0)}) \int_{T}^{T+|x|} \int_{-v}^{T}  |V_u|  |U_u| R^{\frac{3}{2}-\delta}du dv \\
& \lesssim &  e {Z^{(0)}}(M+X^{(0)}+Y^{(0)}+Z^{(0)}) {\left(\int_{T}^{T+|x|} \int_{-v}^{T}  R^{2-2\delta}{|U_u|}^2du dv\right)}^{\frac{1}{2}} \\
& \lesssim &  e {Z^{(0)}}(M+X^{(0)}+Y^{(0)}+Z^{(0)}) {\left(\int_{T}^{T+|x|}  {(v-T)}^{1-2\delta} dv\right)}^{\frac{1}{2}} \\
& \lesssim &  e {Z^{(0)}}(M+X^{(0)}+Y^{(0)}+Z^{(0)}),
\end{eqnarray*}
thus,
\begin{eqnarray*}
{\alpha}_{12}  &\lesssim &      e {X^{(0)}}(M+X^{(0)}+Y^{(0)}+Z^{(0)})+ e {Z^{(0)}}(M+X^{(0)}+Y^{(0)}+Z^{(0)}) \\
&\lesssim &      e ({X^{(0)}}^2+{Y^{(0)}}^2+{Z^{(0)}}^2+M).
\end{eqnarray*}

Before estimating $\eta$, we give a following Morawetz type estimate,
\begin{lemma} \label{lem5.6}
The following estimate holds for $0<T_0<T^*-T_1$,
\begin{equation} \label{5.26}
\int_0^{T_0} \int_0^{T_0-T} {U_u}^2 R^{\sigma-1} dT dR \lesssim \varepsilon
\end{equation}
where $\sigma=\frac{3}{2}$.
\end{lemma}
\begin{proof}
Multiplying \eqref{1.18} by $U_R R^{\sigma}$, we obtain
\begin{eqnarray*}
&&{  \left(U_T U_R R^{\sigma}\right)}_T - {\left( \frac{R^{\sigma}}{2}({U_T}^2+{U_R}^2)  \right)}_R-\left( R^{\sigma-1} {U_R}^2-\frac{\sigma R^{\sigma-1}}{2}({U_T}^2+{U_R}^2) \right)\\
=&&\mathcal{H} \cdot U_R R^{\sigma}+R^{\sigma-1} \frac{2r_u R+r}{r} U_v U_R+R^{\sigma-1} \frac{2r_v R-r}{r} U_u U_R.
\end{eqnarray*}

Integrating this identity with respect to $R$, we get
\begin{eqnarray*}
&& \frac{d}{dT} \int_{0}^{T_0-T} U_T U_R R^{\sigma} dR -\frac{R^{\sigma}}{2} {\left( U_T-U_R \right)}^2  |_{R=T_0-T}-\int_0^{T_0-T} \left(R^{\sigma-1} {U_R}^2-\frac{\sigma R^{\sigma-1}}{2}({U_T}^2+{U_R}^2)\right) dR \\
=&&\int_0^{T_0-T}\mathcal{H}\cdot U_R R^{\sigma}dR+\int_0^{T_0-T}R^{\sigma-1} \frac{2r_u R+r}{r} U_v U_RdR\\
&&+\int_0^{T_0-T}R^{\sigma-1} \frac{2r_v R-r}{r} U_u U_R dR
\end{eqnarray*}
Then, integrating this identity with respect to $T$, we obtain
\begin{eqnarray*}
&&-\tilde{E}(0)-\int R^{\sigma} {|U_u|}^2 du-\int_0^{T_0}\int_0^{T_0-T} (R^{\sigma-1} {U_R}^2-\frac{\sigma R^{\sigma-1}}{2}({U_T}^2+{U_R}^2) )dR dT \\
=&&\int_0^{T_0}\int_0^{T_0-T}\mathcal{H}\cdot U_R R^{\sigma}dR dT+\int_0^{T_0}\int_0^{T_0-T}R^{\sigma-1} \frac{2r_u R+r}{r} U_v U_RdR dT\\
&&+\int_0^{T_0} \int_0^{T_0-T}R^{\sigma-1} \frac{2r_v R-r}{r} U_u U_R dR dT
\end{eqnarray*}
where $\tilde{E}(0)=\int_0^{T_0} U_T U_R R^{\sigma} dR$.

Then, using the fact that $U_T=2U_v-V$, we have
\begin{eqnarray*}
\int_0^{T_0} \int_0^{T_0-T} (\sigma-1){V}^2 R^{\sigma-1} dT dR&=&\tilde{E}(0)+\int R^{\sigma} {|U_u|}^2 du +\int_0^{T_0}\int_0^{T_0-T} (2\sigma R^{\sigma-1}V U_v -2\sigma R^{\sigma-1}{U_v}^2) dR dT \\
&&+\int_0^{T_0}\int_0^{T_0-T}e^{2\lambda}\mathcal{H} \cdot U_R R^{\sigma}dR dT\\
&&+\int_0^{T_0}\int_0^{T_0-T}R^{\sigma-1} \frac{2r_u R+r}{r} U_v U_RdR dT\\
&&+\int_0^{T_0} \int_0^{T_0-T}R^{\sigma-1} \frac{2r_v R-r}{r} U_u U_R dR dT.
\end{eqnarray*}
The integration on the left hand side is nonnegative because $\sigma>1$, and we have
\begin{displaymath}
V \cdot U_v \leq \frac{1}{40\sigma} V^2+10\sigma {U_v}^2.
\end{displaymath}
Then, there exists some constant $C$ such that
\begin{displaymath}
0<C<\sigma-1-\frac{1}{80}.
\end{displaymath}
Therefore,
\begin{eqnarray*}
0 \leq \int_0^{T_0} \int_0^{T_0-T} (\sigma-1-\frac{1}{80}){V}^2 R^{\sigma-1} dT dR &\lesssim & \tilde{E}(0)+\int R^{\sigma} {|U_u|}^2 du+\int_0^{T_0}\int_0^{T_0-T}(20\sigma^2 -2\sigma )R^{\sigma-1}{U_v}^2 dR dT \\
&&+\int_0^{T_0}\int_0^{T_0-T}\phi^2 |V| R^{\sigma}dR dT+\int_0^{T_0} \int_0^{T_0-T} |\phi| |V| R^{\sigma}dR dT \\
&&+\int_0^{T_0}\int_0^{T_0-T}R^{\sigma-1} |\frac{2r_u R+r}{r}| |U_v| |V|dR dT\\
&&+\int_0^{T_0} \int_0^{T_0-T}R^{\sigma-1} |\frac{2r_v R-r}{r}| |U_u| |V| dR dT.
\end{eqnarray*}

Noting that $\tilde{E}(0) \lesssim \varepsilon$ and $\sigma>1$, obviously there holds
\begin{displaymath}
\tilde{E}(0)+\int \frac{R^{\sigma}}{8} {|U_u|}^2 du \lesssim \varepsilon,
\end{displaymath}
then we have
\begin{eqnarray*}
\int_0^{T_0}\int_0^{T_0-T}\phi^2 |V| R^{\sigma}dR dT & \lesssim & \int_0^{T_0}\int_{-v}^{v}\phi^2 |V| R^{\sigma}du dv \\
& \lesssim & {\left(\int_0^{T_0}\int_{-v}^{v}\phi^4  Rdu dv\right)}^{\frac{1}{2}}{\left( \int_0^{T_0}\int_{-v}^{v} V^2 R^{2\sigma-1}du dv\right)}^{\frac{1}{2}} \\
& \lesssim & \varepsilon{\left( \int_0^{T_0}\int_{-v}^{v} {|U_u|}^2 Rdu dv+\int_0^{T_0}\int_{-v}^{v} {|U_v|}^2 R^{2\sigma-1}du dv\right)}^{\frac{1}{2}} \\
& \lesssim & \varepsilon{\left( \varepsilon+\int_0^{T_0}\int_{-v}^{v}  R^{2\sigma-1-2\delta}du dv'\right)}^{\frac{1}{2}} \\
& \lesssim & \varepsilon{\left( \varepsilon+\int_0^{T_0}\int_{-v}^{v}  R^{2-2\delta}du dv\right)}^{\frac{1}{2}}\\
& \lesssim & \varepsilon.
\end{eqnarray*}
Similarly, we have
\begin{displaymath}
\int_0^{T_0} \int_0^{T_0-T} |\phi| |V| R^{\sigma}dR dT \lesssim \varepsilon,
\end{displaymath}
noting that by Lemma \ref{lem5.5},
\begin{displaymath}
|\frac{2R r_u+r}{r} | \lesssim \varepsilon R,\,|\frac{2R r_v-r}{r} | \lesssim \varepsilon R,
\end{displaymath}
then, we obtain
\begin{eqnarray*}
\int_0^{T_0}\int_0^{T_0-T}R^{\sigma-1} |\frac{2r_u R+r}{r}| |U_v| |V|dR dT & \lesssim & \int_0^{T_0}\int_0^{T_0-T}R^{\sigma} |U_v| |V|dR dT \\
& \lesssim & {\left(\int_0^{T_0}\int_0^{T_0-T}R {|V|}^2dR dT\right)}^{\frac{1}{2}} {\left(\int_0^{T_0}\int_0^{T_0-T}R^{2\sigma-1} {|U_v|}^2du dv\right)}^{\frac{1}{2}} \\
& \lesssim & \varepsilon,
\end{eqnarray*}
and
\begin{eqnarray*}
\int_0^{T_0}\int_0^{T_0-T}R^{\sigma-1} |\frac{2r_v R-r}{r}| |U_u| |V|dR dT & \lesssim & \int_0^{T_0}\int_0^{T_0-T}R^{\sigma} |U_u| |V|dR dT \\
& \lesssim & {\left(\int_0^{T_0}\int_0^{T_0-T}R {|V|}^2dR dT\right)}^{\frac{1}{2}} {\left(\int_0^{T_0}\int_0^{T_0-T}R^{2\sigma-1} {|U_u|}^2du dv\right)}^{\frac{1}{2}} \\
& \lesssim & \varepsilon.
\end{eqnarray*}

Now we estimate the left term,
\begin{eqnarray*}
\int_0^{T_0}\int_0^{T_0-T}R^{\sigma-1}{U_v}^2 dR dT & \lesssim & \int_0^{T_0}\int_0^{T_0-T}R^{\sigma-1-2\delta} dR dT \\
&=& \int_0^{T_0}\int_0^{T_0-T}R^{\frac{1}{2}-2\delta} dR dT \\
& \lesssim & \varepsilon.
\end{eqnarray*}
Thus, we have
\begin{displaymath}
\int_0^{T_0}\int_0^{T_0-T}R^{\sigma-1}{U_v}^2 dR dT \lesssim  \varepsilon,
\end{displaymath}
and
\begin{displaymath}
\int_0^{T_0}\int_0^{T_0-T}R^{\sigma-1}{V}^2 dR dT \lesssim  \varepsilon.
\end{displaymath}
Then, by using $U_u=U_v-V$, we finish the proof of the lemma.
\end{proof}

Now we estimate $\eta$. For ${\eta}_1$, we have
\begin{eqnarray*}
{\eta}_1^u  &\lesssim &    \int_{0}^T \int_{-v}^{v} |V_u|\cdot |V| \cdot{|\phi|}^2 R du dv   \\
 & \lesssim & Y^{(1)}  \int_{0}^T \int_{-v}^{v}  |V_u| \cdot{|\phi|}^2 R^{2-\delta} du dv \\
 & \lesssim &  Y^{(1)}Z^{(1)} \int_{0}^{T}dv{\left(   \int_{-v}^{v}  {|\phi|}^4 R^{3-2\delta} du \right)}^{\frac{1}{2}} \\
 & \lesssim & e Y^{(1)}Z^{(1)},
\end{eqnarray*}
and
\begin{eqnarray*}
{\eta}_1^v  &\lesssim &   X^{(1)} Y^{(1)}  \int_{0}^T \int_{-v}^{v} {|\phi|}^2 R^{2-2\delta} du dv   \\
 & \lesssim &  X^{(1)} Y^{(1)} \int_{0}^T dv {\left(   \int_{-v}^{v}  {|\phi|}^4 R du \right)}^{\frac{1}{2}}{\left(   \int_{-v}^{v}   R^{3-4\delta} du \right)}^{\frac{1}{2}} \\
 & \lesssim & e X^{(1)} Y^{(1)},
\end{eqnarray*}
thus,
\begin{displaymath}
{\eta}_1  \lesssim e Y^{(1)}Z^{(1)}+e X^{(1)} Y^{(1)}.
\end{displaymath}
Then, we can get similar estimates for ${\eta}_2,{\eta}_3,{\eta}_4$.

For ${\eta}_{5}$, we have
\begin{eqnarray*}
{\eta}_{5}^u  &\lesssim &    \int_{0}^T \int_{-v}^{v} |V_u|\cdot |W| \cdot{|\phi|}^2 R du dv   \\
 & \lesssim & L^{(1)}  \int_{0}^T \int_{-v}^{v}  |V_u| \cdot{|\phi|}^2 R^{2-\delta} du dv \\
 & \lesssim &  L^{(1)}Z^{(1)} \int_{0}^{T}dv{\left(   \int_{-v}^{v}  {|\phi|}^4 R^{3-2\delta} du \right)}^{\frac{1}{2}} \\
 & \lesssim & e L^{(1)}Z^{(1)},
\end{eqnarray*}
and
\begin{eqnarray*}
{\eta}_{5}^v  &\lesssim &   X^{(1)} L^{(1)}  \int_{0}^T \int_{-v}^{v} {|\phi|}^2 R^{2-2\delta} du dv   \\
 & \lesssim &  X^{(1)} L^{(1)} \int_{0}^T dv {\left(   \int_{-v}^{v}  {|\phi|}^4 R du \right)}^{\frac{1}{2}}{\left(   \int_{-v}^{v}   R^{3-4\delta} du \right)}^{\frac{1}{2}} \\
 & \lesssim & e X^{(1)} L^{(1)},
\end{eqnarray*}
thus,
\begin{displaymath}
{\eta}_{5} \lesssim e (M+{X^{(0)}}^2+{Y^{(0)}}^2+{Z^{(0)}}^2).
\end{displaymath}
Similar estimates hold for ${\eta}_{6}$.

Noting that $|\frac{2R r_u+r}{r} | \lesssim \varepsilon R$, we have
\begin{eqnarray*}
{\eta}_7^v  &\lesssim &   {X^{(1)}}^2  \int_{0}^T \int_{-v}^{v}  R^{1-2\delta} du dv   \\
 & \lesssim & e {X^{(1)}}^2,
\end{eqnarray*}
and
\begin{eqnarray*}
{\eta}_7^u  &\lesssim &   X^{(1)} Z^{(1)}  \int_{0}^T dv {\left(\int_{-v}^{v}  R^{1-2\delta} du \right)}^{\frac{1}{2}}   \\
 & \lesssim & e X^{(1)} Z^{(1)},
\end{eqnarray*}
thus,
\begin{displaymath}
{\eta}_7  \lesssim e {X^{(1)}}^2+e X^{(1)} Z^{(1)}.
\end{displaymath}
Similarly, we get estimates for ${\eta}_{10}$,
\begin{displaymath}
{\eta}_{10} \lesssim e {Z^{(1)}}^2+e X^{(1)} Z^{(1)}.
\end{displaymath}

Then, for $\eta_8$, we have
\begin{eqnarray*}
{\eta}_8^v  &\lesssim &     \int_{0}^T \int_{-v}^{v}  |V_v| \cdot |1+\frac{2r_u R^2}{r^2} r_v-\frac{2r_u R^2}{r^2}r_u| \cdot |U_v| R^{-1}du dv \\
& \lesssim &  \int_{0}^T \int_{-v}^{v}  |V_v| \cdot |1+\frac{2r_u R^2}{r^2} (r_v-r_u-1)+\frac{2r_u R^2}{r^2}| \cdot |U_v| R^{-1}du dv \\
& \lesssim & \varepsilon \int_{0}^T \int_{-v}^{v}  |V_v|  |U_v| R^{\frac{3}{4}}du dv \\
& \lesssim &  e {X^{(1)}} \int_{0}^T \int_{-v}^{v}  R^{\frac{3}{4}-2\delta}du dv \\
& \lesssim &  e {X^{(1)}},
\end{eqnarray*}
then we have
\begin{eqnarray*}
{\eta}_8^u  &\lesssim &     \int_{0}^T \int_{-v}^{v}  |V_u| \cdot |1+\frac{2r_u R^2}{r^2} r_v-\frac{2r_u R^2}{r^2}r_u| \cdot |U_v| R^{-1}du dv \\
& \lesssim & \varepsilon \int_{0}^T \int_{-v}^{v}  |V_u|  |U_v| R^{\frac{3}{4}}du dv \\
& \lesssim &  e {Z^{(1)}} {\left(\int_{0}^T \int_{-v}^{v}  R^{\frac{1}{2}-2\delta}du dv\right)}^{\frac{1}{2}} \\
& \lesssim &  e {Z^{(1)}},
\end{eqnarray*}
thus,
\begin{eqnarray*}
{\eta}_8  &\lesssim &      e {X^{(1)}}+ e {Z^{(1)}} \\
&\lesssim &      e ({X^{(1)}}^2+  {Z^{(1)}}^2+1).
\end{eqnarray*}

Similarly, we can get
\begin{displaymath}
{\eta}_{11}^v   \lesssim   e {X^{(1)}}.
\end{displaymath}
Using Lemma \ref{lem5.6}, we get
\begin{eqnarray*}
{\eta}_{11}^u  &\lesssim &    {Z^{(1)}} {\left(\int_{0}^T \int_{-v}^{v}  R^{\frac{1}{2}} {U_u}^2 du dv\right)}^{\frac{1}{2}} \\
& \lesssim &  \varepsilon {Z^{(1)}},
\end{eqnarray*}
thus,
\begin{eqnarray*}
{\eta}_{11}  &\lesssim &      e {X^{(1)}}+ \varepsilon {Z^{(1)}} \\
&\lesssim &      e ({X^{(1)}}^2+  {Z^{(1)}}^2+1).
\end{eqnarray*}

Now we estimate $\eta_9$ and $\eta_{12}$. For $\eta_9$, we have
\begin{eqnarray*}
{\eta}_9^u  &\lesssim &     \int_{0}^T \int_{-v}^{v}  |V_u| |P| |U_v| du dv \\
& \lesssim & e(M+X^{(0)}+Y^{(0)}+Z^{(0)}) \int_{0}^T \int_{-v}^{v}  |V_u|  |U_v| R^{\frac{3}{2}-\delta}du dv \\
& \lesssim &  e {Z^{(0)}}(M+X^{(0)}+Y^{(0)}+Z^{(0)}) {\left(\int_{0}^T \int_{-v}^{v}  R^{2-4\delta}du dv\right)}^{\frac{1}{2}} \\
& \lesssim &  e {Z^{(0)}}(M+X^{(0)}+Y^{(0)}+Z^{(0)}),
\end{eqnarray*}
and
\begin{eqnarray*}
{\eta}_9^v  &\lesssim &     \int_{0}^T \int_{-v}^{v}  |V_v| |P| |U_v| du dv \\
& \lesssim & e(M+X^{(0)}+Y^{(0)}+Z^{(0)}) \int_{0}^T \int_{-v}^{v}  |V_v|  |U_v| R^{\frac{3}{2}-\delta}du dv \\
& \lesssim &  e {X^{(0)}}(M+X^{(0)}+Y^{(0)}+Z^{(0)}) \int_{0}^T \int_{-v}^{v}  R^{\frac{3}{2}-3\delta}du dv \\
& \lesssim &  e {X^{(0)}}(M+X^{(0)}+Y^{(0)}+Z^{(0)}),
\end{eqnarray*}
thus,
\begin{eqnarray*}
{\eta}_9  &\lesssim &      e {X^{(0)}}(M+X^{(0)}+Y^{(0)}+Z^{(0)})+ e {Z^{(0)}}(M+X^{(0)}+Y^{(0)}+Z^{(0)}) \\
&\lesssim &      e ({X^{(0)}}^2+{Y^{(0)}}^2+{Z^{(0)}}^2+M).
\end{eqnarray*}

For $\eta_{12}$, we have
\begin{eqnarray*}
{\eta}_{12}^v  &\lesssim &     \int_{0}^T \int_{-v}^{v}  |V_v| |Q| |U_u| du dv \\
& \lesssim & e(M+X^{(0)}+Y^{(0)}+Z^{(0)}) \int_{0}^T \int_{-v}^{v}  |V_v|  |U_u| R^{\frac{3}{2}-\delta}du dv \\
& \lesssim &  e {X^{(0)}}(M+X^{(0)}+Y^{(0)}+Z^{(0)}) {\left(\int_{0}^T \int_{-v}^{v}  R^{2-4\delta}du dv\right)}^{\frac{1}{2}} \\
& \lesssim &  e {X^{(0)}}(M+X^{(0)}+Y^{(0)}+Z^{(0)}),
\end{eqnarray*}
then by Lemma \ref{lem5.6},
\begin{eqnarray*}
{\eta}_{12}^u  &\lesssim &     \int_{0}^T \int_{-v}^{v}  |V_u| |Q| |U_u| du dv \\
& \lesssim & (M+X^{(0)}+Y^{(0)}+Z^{(0)}) \int_{0}^T \int_{-v}^{v}  |V_u|  |U_u| R^{\frac{3}{2}-\delta}du dv \\
& \lesssim &   {Z^{(0)}}(M+X^{(0)}+Y^{(0)}+Z^{(0)}) {\left(\int_{0}^T \int_{-v}^{v}  R^{2-2\delta}{|U_u|}^2du dv\right)}^{\frac{1}{2}} \\
& \lesssim &  \varepsilon {Z^{(0)}}(M+X^{(0)}+Y^{(0)}+Z^{(0)}) {\left(\int_{0}^T \int_{-v}^{v}  R^{\frac{1}{2}}{|U_u|}^2du dv\right)}^{\frac{1}{2}} \\
& \lesssim &  \varepsilon {Z^{(0)}}(M+X^{(0)}+Y^{(0)}+Z^{(0)}),
\end{eqnarray*}
thus,
\begin{eqnarray*}
{\eta}_{12}  &\lesssim &      e {X^{(0)}}(M+X^{(0)}+Y^{(0)}+Z^{(0)})+ \varepsilon {Z^{(0)}}(M+X^{(0)}+Y^{(0)}+Z^{(0)}) \\
&\lesssim &      e ({X^{(0)}}^2+{Y^{(0)}}^2+{Z^{(0)}}^2+M).
\end{eqnarray*}

Therefore, we finally obtain
\begin{displaymath}
Z^{(1)} \lesssim M+eX^{(0)}+eY^{(0)}+eZ^{(0)}.
\end{displaymath}

Then, by \eqref{5.22}, we have
\begin{equation*}
Z^{(1)} \lesssim M+eX^{(0)}+eY^{(0)}+eZ^{(1)},
\end{equation*}
which implies
\begin{equation*}
Z^{(1)} \lesssim M+eX^{(0)}+eY^{(0)}.
\end{equation*}

Now we define
\begin{displaymath}
Y^{(1)}_i=\sup\limits_{(T,|x|) \in K_1} r^{\delta-1} \left| \frac{x_i}{R} V \right|
\end{displaymath}
where $x_i$ is the corresponding Cartesian coordinates.

Then, we have
\begin{displaymath}
Y^{(1)} \lesssim \sum_i Y^{(1)}_i,
\end{displaymath}
and there holds
\begin{equation} \label{5.25}
\Box \frac{x_i}{R}V=\frac{x_i}{R}F.
\end{equation}

Recalling the proof of Lemma \ref{lem5.3}, it is not difficult to find that: to estimate $Y^{(1)}_i$ similarly, we only need to estimate ${\left(\int_{-v}^{v} r^{1+2\delta} (\frac{x_i}{R}F_j)^2(u,v) du\right)}^{\frac{1}{2}} $.

Then, we have for $F_1$,
\begin{eqnarray*}
\int_{-v}^{v} r^{1+2\delta} (\frac{x_i}{R}F_1)^2(u,v) du & \lesssim & \int_{-v}^{v} r^{1+2\delta} {|V|}^2 {|\phi|}^4 du \\
& \lesssim & \varepsilon {Y^{(1)}}^2.
\end{eqnarray*}
Similarly, we can get same estimates for $F_2,F_3,F_4$.

For $F_{5}$, we have
\begin{eqnarray*}
\int_{-v}^{v} r^{1+2\delta} (\frac{x_i}{R}F_{5})^2(u,v) du & \lesssim & \int_{-v}^{v} r^{1+2\delta} {|W|}^2 {|\phi|}^4 du \\
& \lesssim & \varepsilon (M+{X^{(0)}}^2+{Y^{(0)}}^2+{Z^{(0)}}^2).
\end{eqnarray*}
Estimates for $F_{6}$ is same.

For $F_7$, we have
\begin{eqnarray*}
\int_{-v}^{v} r^{1+2\delta} (\frac{x_i}{R}F_7)^2(u,v) du & \lesssim & \int_{-v}^{v} R^{2\delta-1} {\left(\frac{2r_u R}{r}+1\right)}^2 {|V_v|}^2 du \\
& \lesssim & \varepsilon {X^{(1)}}^2\int_{-v}^{v} R du \\
& \lesssim & \varepsilon {X^{(1)}}^2.
\end{eqnarray*}
We can estimate $F_{10}$ in a similar way,
\begin{eqnarray*}
\int_{-v}^{v} r^{1+2\delta} (\frac{x_i}{R}F_{10})^2(u,v) du & \lesssim & \int_{-v}^{v} R^{2\delta-1} {\left(\frac{2r_v R}{r}-1\right)}^2 {V_u}^2 du \\
& \lesssim & \int_{-v}^{v} R^{2\delta+1}  {V_u}^2 du \\
& \lesssim & e {Z^{(1)}}^2.
\end{eqnarray*}

For $F_8$, we have
\begin{eqnarray*}
\int_{-v}^{v} r^{1+2\delta} (\frac{x_i}{R}F_8)^2(u,v) du & \lesssim & \int_{-v}^{v} R^{2\delta-3} {|1+\frac{2r_u R^2}{r^2} r_v-\frac{2r_u R^2}{r^2}r_u|  }^2 {|U_v|}^2 du \\
& \lesssim & \int_{-v}^{v} R^{2\delta+\frac{1}{2}}  {|U_v|}^2 du \\
& \lesssim & \int_{-v}^{v} R^{\frac{1}{2}}   du \\
& \lesssim & e,
\end{eqnarray*}
and for $F_{11}$, we obtain
\begin{eqnarray*}
\int_{-v}^{v} r^{1+2\delta} (\frac{x_i}{R}F_{11})^2(u,v) du & \lesssim & \int_{-v}^{v} R^{2\delta+\frac{1}{2}}  {|U_u|}^2 du \\
& \lesssim & \varepsilon.
\end{eqnarray*}

Then, for $F_9$, we have
\begin{eqnarray*}
\int_{-v}^{v} r^{1+2\delta} (\frac{x_i}{R}F_9)^2(u,v) du & \lesssim & \int_{-v}^{v} R^{2\delta-1} {|P|}^2 {|U_v|}^2 du \\
& \lesssim & \int_{-v}^{v} R^{-1} {|P|}^2  du \\
& \lesssim & e(M+{X^{(0)}}^2+{Y^{(0)}}^2+{Z^{(0)}}^2) \int_{-v}^{v} R^{2-2\delta}   du \\
& \lesssim & e(M+{X^{(0)}}^2+{Y^{(0)}}^2+{Z^{(0)}}^2) ,
\end{eqnarray*}
similarly, for $F_{12}$,
\begin{eqnarray*}
\int_{-v}^{v} r^{1+2\delta} (\frac{x_i}{R}F_{12})^2(u,v) du & \lesssim & \int_{-v}^{v} R^{2\delta-1} {|Q|}^2 {|U_u|}^2 du \\
& \lesssim & e(M+{X^{(0)}}^2+{Y^{(0)}}^2+{Z^{(0)}}^2) \int_{-v}^{v} R^{2}  {|U_u|}^2 du \\
& \lesssim & e(M+{X^{(0)}}^2+{Y^{(0)}}^2+{Z^{(0)}}^2).
\end{eqnarray*}

Then we have
\begin{displaymath}
\int_{-v}^{v} r^{1+2\delta} {(\frac{x_i}{r}F_j)}^2(u,v) du  \lesssim M+e{X^{(0)}}^2+e {Y^{(0)}}^2.
\end{displaymath}

Thus, we can get estimates for $Y^{(1)}_i$ as we did in the proof of Lemma \ref{lem5.3}:
\begin{displaymath}
Y^{(1)}_i \lesssim  \left(M+e X^{(0)}+eY^{(0)}\right).
\end{displaymath}
Therefore,
\begin{equation*}
Y^{(1)} \lesssim \sum_i Y^{(1)}_i \lesssim  \left(M+e X^{(0)}+eY^{(0)}\right).
\end{equation*}
Then, by \eqref{5.21}, we have
\begin{equation} \label{5.27}
Y^{(1)}\lesssim  \left(M+e X^{(0)}\right).
\end{equation}

Next, for $X^{(1)}$, rewrite equation \eqref{5.19} in null coordinates,
\begin{displaymath}
-4{V}_{uv}+2\frac{{V}_v-{V}_u}{v-u}=F+\frac{V}{R^2}.
\end{displaymath}

Multiplying by ${R}^{\frac{1}{2}}$, we get
\begin{displaymath}
-4\partial_u \left( {R}^{\frac{1}{2}} {V}_v \right) - \frac{{V}_u}{{R}^{\frac{1}{2}}} =F {R}^{\frac{1}{2}}+\frac{V}{R^{\frac{3}{2}}}.
\end{displaymath}
Integrating the above equation with respect to $u$, we get
\begin{eqnarray*}
-4{|x|}^{\frac{1}{2}} {V}_v (T,|x|)&=&-4{(T+|x|)}^{\frac{1}{2}} {V}_v(0,T+|x|)+\int_{-T-|x|}^{T-|x|} \frac{{V}_u}{{R}^{\frac{1}{2}}}(u,T+|x|) du \\
&&+ \int_{-T-|x|}^{T-|x|}F(u,T+|x|){R}^{\frac{1}{2}}  du+\int_{-T-|x|}^{T-|x|}\frac{V}{R^{\frac{3}{2}}}(u,T+|x|)du.
\end{eqnarray*}

By the regularity of the initial data,
\begin{equation*}
|{(T+|x|)}^{\frac{1}{2}} {V}_v(-v,v)| \leq M{|x|}^{\frac{1}{2}-\delta}.
\end{equation*}

Then, we  only need to consider the last two terms, other terms can be estimated in the same way as we did before. Firstly, we need to estimate $\int_{-T-|x|}^{T-|x|} F_j{R}^{\frac{1}{2}} du$. For $F_1$, we have
\begin{eqnarray*}
\int_{-T-|x|}^{T-|x|} |F_1|{R}^{\frac{1}{2}} du & \lesssim & \int_{-T-|x|}^{T-|x|} |V| \phi^2{R}^{\frac{1}{2}} du \\
& \lesssim &Y^{(1)} \int_{-T-|x|}^{T-|x|} \phi^2{R}^{\frac{1}{2}+1-\delta} du \\
& \lesssim &eY^{(1)}{|x|}^{\frac{1}{2}-\delta} \\
& \lesssim &e(M+eX^{(0)}){|x|}^{\frac{1}{2}-\delta}.
\end{eqnarray*}
We can get same estimates for $F_2,F_3,F_4$.

For $F_{5}$,
\begin{eqnarray*}
\int_{-T-|x|}^{T-|x|} |F_{5}|{R}^{\frac{1}{2}} du & \lesssim & \int_{-T-|x|}^{T-|x|} |W| \phi^2{R}^{\frac{1}{2}} du \\
& \lesssim &L^{(1)} \int_{-T-|x|}^{T-|x|} \phi^2{R}^{\frac{1}{2}+1-\delta} du \\
& \lesssim &eL^{(1)}{|x|}^{\frac{1}{2}-\delta} \\
& \lesssim &e(M+eX^{(0)}){|x|}^{\frac{1}{2}-\delta}.
\end{eqnarray*}
Estimates for $F_{6}$ is same.

Then, for $F_7$ we have
\begin{eqnarray*}
\int_{-T-|x|}^{T-|x|} |F_7|{R}^{\frac{1}{2}} du & \lesssim & \int_{-T-|x|}^{T-|x|} |V_v| |\frac{2r_uR}{r}+1|{R}^{-\frac{1}{2}} du \\
& \lesssim &eX^{(1)} \int_{-T-|x|}^{T-|x|} {R}^{-\frac{1}{2}-\delta} du \\
& \lesssim &eX^{(0)}{|x|}^{\frac{1}{2}-\delta}.
\end{eqnarray*}
Similarly, for $F_{10}$ we have
\begin{eqnarray*}
\int_{-T-|x|}^{T-|x|} |F_{10}|{R}^{\frac{1}{2}} du & \lesssim & \int_{-T-|x|}^{T-|x|} |V_u| |\frac{2r_vR}{r}-1|{R}^{-\frac{1}{2}} du \\
& \lesssim &e \int_{-T-|x|}^{T-|x|} |V_u| R du \\
& \lesssim &eZ^{(1)}\\
& \lesssim &eZ^{(1)}{|x|}^{\frac{1}{2}-\delta}\\
& \lesssim &e(M+eX^{(0)}){|x|}^{\frac{1}{2}-\delta}.
\end{eqnarray*}

For $F_8$, there holds
\begin{eqnarray*}
\int_{-T-|x|}^{T-|x|} |F_8|{R}^{\frac{1}{2}} du & \lesssim & \int_{-T-|x|}^{T-|x|} |U_v| |-\frac{1}{R^2}-\frac{2r_u}{r^2}(r_v-r_u)|{R}^{\frac{1}{2}} du \\
& \lesssim &\int_{-T-|x|}^{T-|x|}  |R^2\cdot(-\frac{1}{R^2}-\frac{2r_u}{r^2}(r_v-r_u))|{R}^{-\frac{3}{2}-\delta} du  \\
& \lesssim &\varepsilon \int_{-T-|x|}^{T-|x|}  {R}^{-\frac{1}{2}-\delta} du \\
& \lesssim &e{|x|}^{\frac{1}{2}-\delta}.
\end{eqnarray*}
Similarly, for $F_{11}$, we can get
\begin{eqnarray*}
\int_{-T-|x|}^{T-|x|} |F_{11}|{R}^{\frac{1}{2}} du & \lesssim &  \int_{-T-|x|}^{T-|x|} |U_u| |R^2\cdot(\frac{1}{R^2}-\frac{2r_v}{r^2}(r_v-r_u))|{R}^{-\frac{3}{2}} du  \\
& \lesssim &\varepsilon \int_{-T-|x|}^{T-|x|}  {R}^{\frac{1}{4}}|U_u| du \\
& \lesssim &\varepsilon {\left(\int_{-T-|x|}^{T-|x|}  {R}^{-\frac{1}{2}} du\right)}^{\frac{1}{2}} \\
& \lesssim &\varepsilon {\left(\int_{-T-|x|}^{T-|x|}  {R}^{-2\delta}du\right)}^{\frac{1}{2}} \\
& \lesssim &e{|x|}^{\frac{1}{2}-\delta}.
\end{eqnarray*}

Finally, for $F_9$ we have
\begin{eqnarray*}
\int_{-T-|x|}^{T-|x|} |F_9|{R}^{\frac{1}{2}} du & \lesssim & \int_{-T-|x|}^{T-|x|} |U_v| |P|{R}^{-\frac{1}{2}} du \\
& \lesssim &e(M+X^{(0)}+Y^{(0)}+Z^{(0)}) \int_{-T-|x|}^{T-|x|} R^{-\frac{1}{2}-\delta} du \\
& \lesssim &e(M+X^{(0)}+Y^{(0)}+Z^{(0)}){|x|}^{\frac{1}{2}-\delta}\\
& \lesssim &(M+eX^{(0)}){|x|}^{\frac{1}{2}-\delta},
\end{eqnarray*}
and similarly for $F_{12}$,
\begin{eqnarray*}
\int_{-T-|x|}^{T-|x|} |F_{12}|{R}^{\frac{1}{2}} du & \lesssim & \int_{-T-|x|}^{T-|x|} |U_u| |Q|{R}^{-\frac{1}{2}} du \\
& \lesssim &e(M+X^{(0)}+Y^{(0)}+Z^{(0)}) \int_{-T-|x|}^{T-|x|} |U_u| R^{1-\delta} du \\
& \lesssim &e(M+X^{(0)}+Y^{(0)}+Z^{(0)}){\left(\int_{-T-|x|}^{T-|x|}  R^{1-2\delta} du\right)}^{\frac{1}{2}}\\
& \lesssim &e(M+X^{(0)}+Y^{(0)}+Z^{(0)}){\left(\int_{-T-|x|}^{T-|x|}  R^{-2\delta} du\right)}^{\frac{1}{2}}\\
& \lesssim &e(M+X^{(0)}+Y^{(0)}+Z^{(0)}){|x|}^{\frac{1}{2}-\delta}\\
& \lesssim &(M+eX^{(0)}){|x|}^{\frac{1}{2}-\delta}.
\end{eqnarray*}

For the last term, we have
\begin{eqnarray*}
\int_{-T-|x|}^{T-|x|}\frac{V}{R^{\frac{3}{2}}}du
& \lesssim &Y^{(1)} \int_{-T-|x|}^{T-|x|} R^{-\frac{1}{2}-\delta} du \\
& \lesssim &(M+eX^{(0)}) \int_{-T-|x|}^{T-|x|} R^{-\frac{1}{2}-\delta} du\\
& \lesssim &(M+eX^{(0)}){|x|}^{\frac{1}{2}-\delta}.
\end{eqnarray*}

Thus, as we did before to estimate $X$, we can obtain
\begin{equation*}
X^{(1)} \leq M+e X^{(0)}.
\end{equation*}
By \eqref{5.20}, we have
\begin{equation*}
 X^{(1)} \leq M+e X^{(1)},
\end{equation*}
which implies
\begin{equation*}
 X^{(1)} \leq M.
\end{equation*}

Thus, we finally obtain
\begin{equation} \label{5.28}
X^{(0)} \leq M.
\end{equation}
Then, by a result of Christodoulou-Tahvildar-Zadeh \cite{christodoulou2}, it follows that the full gradients $\partial U$ is bounded.

By \eqref{5.23} and \eqref{5.24}, we have
\begin{equation} \label{5.29}
\sup\limits_{K_0} r^{\delta} |W_v| \leq M,
\end{equation}
and
\begin{equation} \label{5.30}
|W|=|\lambda_R| \leq M r^{1-\delta} \leq M.
\end{equation}

Using Lemma \ref{remark3.10} and \eqref{5.29}, we have
\begin{eqnarray} \label{5.31}
|\lambda_v (T,R)| &\leq& \int_0^R |W_v(T,R')| dR' \\ \nonumber
&\leq& M\int_0^R {(R')}^{-\delta} dR' \\ \nonumber
&\leq& M.
\end{eqnarray}

Now we have shown that the full gradients of all unknowns are bounded. Then the higher regularity of the unknowns follows using energy estimates. Thus, the proof of Theorem \ref{thm5.1} is finished now.

\section*{Acknowledgement}
Both authors are grateful to Prof. Naqing Xie for fruitful discussions. The first author especially thanks him for his kind guidance.

Y. Zhou was supported by Key Laboratory of Mathematics for Nonlinear Sciences (Fudan University), Ministry of Education of China, P.R.China. Shanghai Key Laboratory for Contemporary Applied Mathematics, School of Mathematical Sciences, Fudan University, P.R. China, NSFC (grants No. 11421061, grants No.11726611, grants No. 11726612), 973 program (grant No. 2013CB834100) and 111 project.

\end{document}